\documentclass[12pt,reqno]{amsart} 
\makeatletter
\def\ps@headings{%
  \def\@evenhead{\hfil\scriptsize\normalfont\leftmark}%
  \def\@oddhead{\scriptsize\normalfont\rightmark\hfil}%
  \def\@evenfoot{\hfil\thepage\hfil}%
  \def\@oddfoot{\hfil\thepage\hfil}%
}
\makeatother

\usepackage[a4paper, margin = .8in]{geometry} 
\usepackage{bbm}
\usepackage{mathtools}
\usepackage[colorlinks=true,linkcolor=RoyalBlue, citecolor=MidnightBlue, urlcolor=RoyalBlue]{hyperref}
\numberwithin{equation}{section}
\usepackage{lmodern} 
\usepackage{mathrsfs} 
\usepackage{graphicx} 
\usepackage{amsfonts,amsmath,amssymb,amsthm}
\usepackage[dvipsnames]{xcolor}
\usepackage{bbold}
\usepackage{framed}
\usepackage{thmtools}
\usepackage[
backend=biber,
style=alphabetic,
sorting=nyt,
maxnames=100,
giveninits=true
]{biblatex}
\usepackage{framed}
\newtheorem{thm}{Theorem}[]

\newtheorem{lemma}{Lemma}[section]
\newtheorem{notation}{Notation}[section]

\theoremstyle{remark}
\newtheorem{remark}{Remark}
\theoremstyle{definition}
\newcommand{\smc}{\sum\cdots\sum}
\newcommand{\sumstar}{\sideset{}{^*}\sum}

\newcommand{\barr}{\overline}
\newcommand{\nn}{\nonumber}

\newcommand{\Fm}{\mathfrak{m}}

\newcommand{\BC}{{\mathbb {C}}}

 \newcommand{\BN}{{\mathbb {N}}}
 
 \newcommand{\bR}{{\mathbb {R}}}

 \newcommand{\BZ}{{\mathbb {Z}}}
 
\newcommand{\CC}{{\mathcal {C}}}
 
 \newcommand{\CH}{{\mathcal {H}}}
\newcommand{\CI}{{\mathcal {I}}} 
 
\newcommand{\CM}{{\mathcal {M}}} \newcommand{\CN}{{\mathcal {N}}}
 
\newcommand{\cC}{{\mathcal {C}}}

 \newcommand{\CX}{{\mathcal {X}}}

 \newcommand{\GL}{{\mathrm{GL}}}

\newcommand{\fm}{{\mathfrak{m}}}

\newcommand{\fc}{{\mathfrak{c}}} \newcommand{\fd}{{\mathfrak{d}}}

\renewcommand{\Re}{{\ \mathrm{Re}\ }}

\newcommand{\SL}{{\mathrm{SL}}}

\usepackage{comment}

\title{On shifted convolution sums of $\GL(3)$-Fourier coefficients with an average over shifts}
 \author{Ritwik Pal}
\address{Indraprastha Institute of Information Technology, Delhi, India}
\email{ritwik.1729@gmail.com}
 \author{Sampurna Pal}
\address{Theoretical Statistics and Mathematics Unit, Indian Statistical Institute, Kolkata, WB, India}
\email{sampu.andul@gmail.com}
\begin{document}

\begin{abstract}
 Let $F$ be a Hecke-Maass cusp form for $\SL_3(\mathbb{Z})$ and $A(m,n)$ be its normalized Fourier coefficients. Let $V$ be a smooth function, compactly supported on $[1,2]$ and satisfying $V(y)^{j} \ll_j y^{-j}$ for any $j \in \mathbb{N} \cup \{0\}$. In this article we prove a power-saving upper bound for the `average' shifted convolution sum 
\begin{equation*}\label{MainEquation}
\sum_{h}\sum_{n}A(1,n)A(1,n+h)V\left(\frac{n}{N}\right)V\left(\frac{h}{H}\right),
\end{equation*}
for the range $N^{1/2-\varepsilon} \geq H \geq N^{1/6+ \varepsilon}$, for any $\varepsilon >0$. This is an improvement over the previously known range $N^{1/2-\varepsilon} \geq H \geq N^{1/4+ \varepsilon}$.
\end{abstract}
\subjclass[2010]{11F66 (primary), 11M41 (secondary)}
\keywords{Shifted convolution sum, Delta method, $L$-functions}
\maketitle
\section{Introduction}
Convoluted summation of two distinct sequences has been a long, well studied theme in Analytic Number theory. One of the avenues in this theme is the study of shifted convolution sum problems
\begin{equation*}
    \sum_{n \leq N} a(n)b(n+h),
\end{equation*}
 for two arithmetic functions $a(n)$ and $b(n)$. When the sequences come from the coefficients of automorphic $L$-functions, its importance appears due to its implications on to the problems on equidistribution (QUE), sub-convexity problems of the associated $L$-functions (see    \cite{Blomer_shifted_IMRN}, \cite{blomer_spectral_2008}, \cite{Harcos_add}, \cite{holowinsky_sieve_2009},  \cite{kowalski_rankin-selberg_2002}, \cite{luo_mass_2003}, \cite{michel_subconvexity_2004}, \cite{sarnak_estimates_2001}) and etcetera. The study of shifted convolution sums on automorphic coefficients was started in (see \cite{ingham1927}) for the sequence $d(n)d(n+h)$, giving its asymptotic expression. Later, subsequent improvements on its error term was made in \cite{Estermann1931}, \cite{heathbrownfourth} and \cite{Des_iwa_additivediv}. The article of \cite{Des_iwa_additivediv} was first to invoke the spectral theory of automorphic forms for shifted convolution sums.
 For a holomorphic modular form $f$, considering a `smooth-version' of shifted convolution sum
     \[\sum_{n}\lambda_f(n)\lambda_f(n+h)V\left(\frac{n}{N}\right),\]
    where $V$ is some smooth function, the first power-saving upper bound was shown in \cite{good_1983}.  W. Duke, J. Friedlander, H. Iwaniec \cite{duke_quadratic_1994} first took an alternate way through their `delta-method' approach. For shifted convolution sum of Maa\ss~ forms, M. Jutila in \cite{jutila_additive_1996_1}, \cite{jutila_additive_1997} proved upper bounds through some remarkable approaches using the spectral theory of automorphic forms with some crucial inputs from earlier works by H. Iwaniec. It is worth mentioning here that so far to treat the shifted convolution sum problems for cuspidal automorphic coefficients, one considers either the spectral theory approach or the delta-method approach.   %

 Shifted convolution sum problems for $\GL(3) \times \GL(2)$ coefficients were considered in \cite{Pittd3n} and
     \cite{Munshi_shifted_duke}. For $\GL(3)\times \GL(3)$ shifted convolution sums it is a hard open problem to obtain cancellation only in the shifted convolution sum with a fixed shift $h$. So one considers an additional average over the shifts:
\begin{equation*}
    \sum_{h\leq H}\sum_{n \leq N} a_{\pi_1}(n) a_{\pi_2}(n+h),
\end{equation*}
where $\pi_1, \pi_2$ are two $\GL(3)$-forms. The aim is to show some non-trivial upper bounds for these sums while the size of $H$ is as small as possible in comparison to $N$. The first result in this direction was proved in \cite{Baier_Browning_Marasingha_Zhao_2012}. They
        proved asymptotic estimate with power-saving error term for the shifted convolution sum 
        $$\sum_{h\leq H}\sum_{N< n\leq 2N}d_3(n)d_3(n+h),$$
        for $N^{1/6+\varepsilon}\leq H \leq N^{1-\varepsilon}$.
        However, their use of results in moment of the Riemann zeta function is not available in the cuspidal case.

Let $F$ be a Hecke-Maass cusp form for $\SL_3(\mathbb{Z})$ and $A(m,n)$ be its normalized Fourier coefficients. In this article, we consider the `average' shifted convolution sum 
\begin{equation}\label{MainEquation}
    S':=\sum_{h}\sum_{n}A(1,n)A(1,n+h) V\left(\frac{n}{N}\right)V\left(\frac{h}{H}\right),
\end{equation}
 for some smooth function $V$ that is compactly supported on $[1,2]$ and satisfy $V(y)^{j} \ll_j y^{-j}$ for any $j \in \mathbb{N} \cup \{0\}$. Ramanujan-Petersson bound on average implies that it is trivially bounded by $(NH)^{1+\varepsilon}$. With application of delta method \cite{harun-Singh} proved some non-trivial cancellation when $H\geq N^{1/2+ \varepsilon}$. Thereafter, recently \cite{dasgupta2024secondmomentgl3standard} made a significant improvement and showed that a non-trivial upper bound for the range $H \geq N^{1/4+ \varepsilon}$. Our main aim for this article is to prove a non-trivial upper bound for even smaller $H$ with respect to $N$, i.e., $H \geq N^{1/6+ \varepsilon}$. 
\begin{thm}\label{maintheorem}
  Let $S'$ be as in \eqref{MainEquation}. Then for any $H\leq N^{1/2-\varepsilon}$, we have 
    \begin{align}
        \sum_{h}\sum_{n}A(1,n)A(1,n+h)V\left(\frac{h}{H}\right)V\left(\frac{n}{N}\right)\ll_{F,\varepsilon}  N^{\varepsilon}\left(N^{1/4}H^{5/2}+H^{5/8}N^{17/16}+N^{9/8}H^{1/4}\right).
    \end{align}
\end{thm}
\begin{remark}In particular, this bound is a power-saving improvement over the trivial bound $(NH)^{1+\varepsilon}$ for $H\geq  N^{1/6+\varepsilon}$. 
\end{remark}
Let us now briefly discuss the route of the methods of this article.
We begin by a well-known approach of separating the oscillations in (\ref{MainEquation}) by the delta method of Duke-Friedlander-Iwaniec along with the conductor lowering mechanism of Munshi to arrive at 
\begin{align}S'\asymp \int_{x\sim 1}\frac{1}{C}\sum_{q\sim C}\frac{1}{q}\ \sumstar_{a\bmod{q}} \sum_{h\sim H}\sum_{n\sim N}\sum_{m\sim N}A(n)A(m)e\left(\frac{a(m-n-h)}{q}\right)e\left(\frac{x(m-n-h)}{qC}\right),\label{Sjustafterdfi}\end{align}
which is trivially of the size $N^2H$.
We choose the conductor of the delta symbol expansion (Lemma \ref{delta}) to be $C=\sqrt{N/K}$, where $K\gg 1$, which introduces an extra analytic oscillation roughly of the size $K$ in terms of $e\left(\frac{x(m-n-h)}{qC}\right)$. Then we apply the Poisson summation formula on the $h$-sum and the Voronoi-type summation formula for $\SL(3,\mathbb{Z})$ on the $n$-sum and the $m$-sum to arrive at 
\begin{equation}\label{Safterdualiation}
 S'\asymp\frac{H}{K}\int\limits_{x\sim 1}\frac{1}{C^2}\sum_{q\sim C}\sum_{h\sim \frac{C}{H}}\sum_{n\sim  \frac{N^2}{C^3} }\sum_{m\sim  \frac{N^2}{C^3} }A(n)A(m)S(\overline h, n;q)S(\overline h, m;q) e\left(\frac{2\sqrt{C}(\sqrt{n}-\sqrt{m})}{q\sqrt{x}}\right).
\end{equation} At this point, $S'$ can be trivially bounded by $N^{3/2}K^{3/2}$. Comparing it with the trivial bound in (\ref{Sjustafterdfi}), we note that by dualization, we have saved $\frac{N^2H}{N^{3/2}K^{3/2}}\sim \frac{\sqrt{N}H}{K^{3/2}}.$

At this point, we intend to apply the duality principle as in \cite{ALM}, \cite{dasgupta2024secondmomentgl3standard} and \cite{pal2023secondmomentdegreelfunctions}  to interchange the order of the summations. 
But preceding that,  we will `linearize' the phase function $\frac{2\sqrt{Cn}}{q\sqrt{x}}$ in (\ref{Safterdualiation}), which can be considered a novelty of this paper. It has been one of the crucial inputs to get an extra-saving on the upper bound of $S'$. 

Looking closely at the $x$-integral
\begin{align}
\int_{x\sim 1}e\left(\frac{2\sqrt{C}(\sqrt{n}-\sqrt{m})}{q\sqrt{x}}\right) dx,\label{xintsketch}
\end{align}
we observe that this integral is negligibly small unless $$\sqrt{n}-\sqrt{m}\ll \sqrt{C} \iff |n-m|\ll \sqrt{C}\cdot \frac{N}{C^{3/2}}\sim CK. $$
In that range we can absorb the above term into the smooth function and `treat' it like $1$. So in principle, this integral acts like $\delta(|n-m|\ll CK)$ and we should be able to `replace' the above integral by 
\begin{align}CK\int_{y\sim \frac{1}{CK}}e((n-m)y).\label{yintsketch}\end{align}
The advantage of this linearization of the phase function is two-fold. Firstly, when we apply the Poisson summation formula on the $n$-sum after the application of the duality principle, we can estimate the resulting integral transform by the first derivative bound instead of stationary phase analysis. This results in a saving of the full size of the oscillation $K$ instead of its square-root. Secondly, treating the integral transform by the first derivative bound leaves no analytic oscillation for the subsequent steps, which reduces the overall conductor and makes the integral transforms arising from the Poisson summation formulas at the succeeding steps much simpler. 

However, significant challenge arises from the smooth function of the $y$-integral (\ref{yintsketch}), which originates from the smooth function ($g(q,x)$) of the $x$-integral (\ref{xintsketch}) as it involves the variables $m,n,q$ and $y$ in a non-separable manner. We treat this in Section \ref{subsec.xint} and Appendix \ref{appen.lin} through Mellin-Barnes formula and make it conducive for the application of the AM-GM inequality and the duality principle.  

After the linearization of the phase function and the application of the AM-GM inequality, we arrive at 
\begin{equation}
    S'\asymp\frac{H}{C}\int_{y\sim \frac{1}{CK}}\sum_{q\sim C}\sum_{h\sim C/H}\Big|\sum_{m\sim  N^2/C^3 }A(m)e(my)S(\overline h, m;q)\Big|^2.
\end{equation}
Then if we apply the duality principle, we can bound $S'$ by 
\begin{align}
    S'\ll \frac{H}{C} \Big(\sum_{m\sim N^2/C^3}|A(n)|^2\Big)\cdot \Delta\ll HK^2\Delta,
\end{align}
where \begin{align}\Delta=\sup_{||\alpha||_2=1}\sum_{m\sim  \frac{N^2}{C^3} }\Big|\int_{y\sim \frac{1}{CK}}\sum_{q\sim C}\sum_{h\sim \frac{C}{H}}\alpha(q,h,x)e(my)S(\overline h, m;q)\Big|^2.\label{Deltaatintrofirst}\end{align}
At this point, $\Delta$ can be trivially bounded by $N^2/HCK$ and $S'\ll \frac{N^2K}{C}$. We then open up the absolute square and apply the Poisson summation formula on the $m$-sum
\begin{align*}
    T_\fm=\sum_{m\sim N^2/C^3}S(\bar h_1, m;q_1)S(\bar h_2, m;q_2)e(m(y_1-y_2)).
\end{align*} Due to our linearization of the phase function, the resulting integral transform is of the form 
$$\int_{y\sim  \frac{1}{CK} }e(y(y_1-y_2)e\left(-\frac{my}{q_1q_2}\right)dy,$$
where the dual variable is also denoted by $m$. We treat it by repeated integration by parts (Lemma \ref{repeatedint}), which imply that the integral is negligibly small outside the range 
$$|C^2(y_1-y_2)-m|\ll C^5/N^2 \text{ and } |m|\ll C/K.$$
Inside this range, the $y$-integral can be treated as a mildly-oscillatory(i.e., $x^{j}V^j(x)\ll_{\varepsilon} N^{j\varepsilon}$) function, essentially making the variables $y_1$ and $y_2$ to be free. This condition can be considered as a restriction on the length of either the $y_1$-integral or the $y_2$-integral, instead of a restriction on the $m$-sum. At the end, we have 
\begin{align}
    T_\fm\asymp \frac{N^2}{C^3}\sum_{m\ll C/K} \CC\cdot  W\Big(\frac{C^2(y_1-y_2)-m}{C^5/N^2}\Big),
\end{align}
where $$\CC=\frac{1}{q_1q_2}\sum_{\beta \bmod q_1q_2}S(\bar h_1,\beta, q_1)S(\bar h_2, \beta, q_2)e_{q_1q_2}(m\beta).$$
At the zero frequency ($m=0$),
we compute that (see \ref{Deltaatintrofirst}) $\Delta \ll C$ and it leads to (see (\ref{Sdiagonal})) 
\begin{align}S'_{m=0}\ll HN^{1/2}K^{3/2},\label{S_sketch_m0}\end{align}
which is smaller than the trivial bound if $K\ll N^{1/3}$. 

For the nonzero frequencies ($m\neq 0$), by reciprocity as in \cite{pal2023secondmomentdegreelfunctions}, we evaluate that 
$$\CC\asymp e\left(\frac{q_1\bar q_2}{mh_2}\right)e\left(\frac{\bar q_1q_2}{mh_1}\right).$$
At this point, for $m\neq 0$, $S$ can be bounded above be $\frac{N^{3/2}}{K^{1/2}}$. Now applying Cauchy's inequality on $\Delta$, keeping only the $q_2$-sum inside, we bound $\Delta$ by 
\begin{align}
    \Delta\ll \frac{N^2}{C^3}S_0^{1/2}S_1^{1/2},\label{DeltaOneCauchy}
\end{align}
where 
\begin{align}
    S_0=&\int_{y_1}\sum_{h_1\sim \frac{C}{H}}\sum_{q_1\sim C}|\alpha_1|^2\sum_{h_2\sim \frac{C}{H}}\int_{y_2} \sum_{m\ll \frac{C}{K}} W\left(\frac{C^2(y_1-y_2)-m}{C^5/N^2}\right)
    \ll\frac{C}{HK^3},
\end{align}
and 
\begin{align}
    S_1=&\int_{y_1}\sum_{h_1\sim \frac{C}{H}}\sum_{h_2\sim \frac{C}{H}}\int_{y_2} \sum_{m\ll \frac{C}{K}} W\left(\frac{C^2(y_1-y_2)-m}{C^5/N^2}\right) \sum_{q_2\sim C}\alpha_2\sum_{q_3\sim C}\alpha_{3}\cdot S_q, \nonumber \\
    \text{where}~~~~~~S_q:=&\sum_{q_1\sim C} e\left(\frac{q_1(\bar q_2-\bar q_3)}{mh_2}\right)e\left(\frac{\bar q_1(q_2-q_3)}{mh_1}\right)
    \text{ and } S_q\ll \frac{C^3}{HK^3}.\nonumber
\end{align}
In the diagonal case $q_2=q_3$, we trivially bound $S_1\ll \frac{N}{K^4H}$. Hence, we bound $S$ by 
\begin{align}S'_{q_2=q_3}\ll \frac{N^{5/4}}{K^{1/4}}.\label{S_sketch_q2=q3} \end{align}
When $q_2\neq q_3$, we apply the Poisson summation formula on the $q_1$-sum $S_q$ and get 
\begin{align}
    S_q=\frac{C}{mh_1}\sum_{q_1^*\ll \frac{mh_1h_2}{C}}S((q_2-q_3), (\bar q_2-\bar q_3)h_1\bar h_2-q_1^*\bar{h_2};mh_1)\delta(\bar q_2-\bar q_3\equiv -q_1^*\overline{h_1}\bmod h_2)\nonumber,
\end{align}
where the dual variable is denoted by $q_1^*$. At the zero frequency $q_1^*=0$, the Kloosterman sum splits into a Kloosterman sum modulo $m$ and a Ramanujan sum modulo $h_1$. So we get $ S_q\ll \frac{C}{\sqrt{m}}\delta(q_2\equiv q_3\bmod h_1 h_2)$ and consequently, we can bound $S_1$ by 
$$S_1\ll \frac{C^3}{N^2}\cdot \frac{C^{1/2}}{K^{1/2}}\cdot C\cdot \frac{C}{H}\left(\frac{C}{(C/H)^2}+1\right)\ll \frac{N^{1/4}H}{K^{11/4}}+\frac{N^{3/4}}{K^{13/4}H}.$$
As $K\leq N^{1/3}$, the second term is smaller than the contribution of $S_1$ from the case $q_2=q_3$. Thus, the contribution of the first term in $S'$ can be bounded by 
\begin{align}
    S'_{q_1^*=0}\ll HK^2\cdot \frac{N^2}{C^3}\cdot \left(\frac{C}{HK^3}\cdot \frac{N^{1/4}H}{K^{11/4}}\right)^{1/2}\ll HN^{7/8}K^{3/8}. \label{S_sketchq_1=0}
\end{align}
For the nonzero frequencies $q_1^*\neq 0$, we have $ S_q\ll \sqrt{mh_1}$ and $S_1\ll \frac{C^3}{K^{7/2}H^{3/2}}$. Consequently we have $S_{q_1^*\neq 0}\ll \frac{N^{3/2}}{H^{1/4}K^{3/4}}$.
\begin{remark}
    At this point, with the choice of $K=\frac{N^{5/9}}{H^{10/9}}$, we have 
    $$S'\ll N^{\varepsilon}(N^{13/12}H^{7/12}+N^{10/9}H^{5/18}).$$
    This bound is already non-trivial in the range $N^{1/5+\varepsilon}\leq H\leq N^{1/2-\varepsilon}$, which would have been an improvement over the existing best known range of  $N^{1/4+\varepsilon}\leq H\leq N^{1/2-\varepsilon}$ as in \cite{dasgupta2024secondmomentgl3standard}. 
\end{remark}

To get a finer estimate for the nonzero frequencies, we apply infinite Cauchy's inequality as in \cite{ALM}. Here, we apply Cauchy's inequality for $j$-many times on $S_1$ and get 
$$S_1\ll \prod_{i=1}^{j-1}S_{i,1}^{1/2^i}\cdot S_{j}^{1/2^{j}},$$
where in every step $S_{j-1}\mapsto S_{j}$, for $j\geq 2$, we first apply Cauchy's inequality on $S_{j-1}$ keeping only the $q_{j+1}$-sum inside. Very crudely, after the application of the Poisson summation formula on $q_{j-1}$, $S_{j-1}$ is of the form 
$$S_{j-1}\rightsquigarrow\cdots \sum_{q_{j-1}^*}\cdots \sum_{q_j}\alpha_j\sum_{q_{j+1}}\alpha_{j+1}\cdots .$$
Once we apply the Cauchy's inequality on $S_{j-1}$ keeping only the $q_{j+1}$-sum inside, we get $S_{j-1}\leq S_{1,j-1}^{1/2}S_j^{1/2}$ where 
$$S_{1,j-1}\rightsquigarrow \cdots \sum_{q_{j-1}^*}\cdots \sum_{q_j}|\alpha_j|^2~~~\text{ and }~~~S_{j}\rightsquigarrow\cdots \sum_{q_{j-1}^*}\cdots \sum_{q_j}\Big|\sum_{q_{j+1}}\alpha_{j+1}\cdots\Big|^2 .$$ We bound $S_{1,j-1}$ trivially and in $S_j$ we open the absolute square and apply Poisson summation formula on $q_j$ and repeat the process. At every step, the variable $q_j$ is of the size $C/h_2$ and the arithmetic conductor is of the size $mh_1$. Then in every step, due to the Poisson summation formula, we save $\frac{C}{\sqrt{mh_1}h_2}\sim  \frac{H^{3/2}\sqrt{K}}{C}$. Thus, after $j$-many applications of the Cauchy's inequality, we can bound $S_1$ by $$S_1\ll \frac{C^3}{K^{7/2}H^{3/2}} \cdot \Big(\frac{C}{H^{3/2}\sqrt{K}}\Big)^{\sum_{i=1}^j2^{-j}}.$$
Thus, by applying Cauchy's inequality ad infinitum, we get $S_1 \ll \frac{C^4}{K^4H^3}.$
Hence, we derive 
\begin{align}
    S'_{q_1^*\neq 0}\ll HK^2\cdot \frac{N^2}{C^3}\cdot \left(\frac{C}{HK^3}\cdot \frac{C^4}{K^4H^3}\right)^{1/2} \ll \frac{N^{7/4}}{K^{5/4}H}.\label{S_sketch_q1nonzero}
\end{align}
Finally, from (\ref{S_sketch_m0}), (\ref{S_sketch_q2=q3}), (\ref{S_sketchq_1=0}) and (\ref{S_sketch_q1nonzero}) and choosing $K=\frac{\sqrt{N}}{H}$, we derive Theorem \ref{maintheorem}. 

\section{Preliminaries}
\subsection{$\GL(3)$ Maa\ss~ forms}
Let $F$ be a normalized Hecke-Maa\ss~ cusp form for $\SL(3,\BZ)$ with Fourier coefficient $A(n_1,n_2)$. Then the $L$-function associated to $F$ can be defined by the Dirichlet series 
    \begin{align}
        L(F,s):=\sum_{n_2=1}^{\infty}A(1,n_2)n_2^{-s}, ~~~\text{for}~~\Re(s)>1.\label{Lfunctiondefn}
    \end{align}

\begin{lemma}[Ramanujan bound on average, \cite{molteni_upper_2002}]\label{rambound} 
Let the notations be as in above, the Fourier coefficients $A(n_1, n_2)$ satisfy
    \begin{equation}
        \sum_{n_1^2n_2\ll N}|A(n_1,n_2)|^2\ll_{F, \varepsilon} N^{1+\varepsilon}.
    \end{equation}
\end{lemma} 
\begin{lemma}\label{anotherramanujanbound}
The Fourier coefficients $A(n_1,n_2)$ satisfies the following: 
    $$\sum_{n_2\sim \frac{N}{n_1^2}}|A(n_1,n_2)|^2\leq \frac{N^{1+\varepsilon}}{n_1^2}\sum_{d|n_1}\frac{|A(n_1/d,1)|^2}{d}.$$
\end{lemma}
\begin{proof}
Using the standard Hecke-relation, we derive  
    \begin{align*}
        \sum_{n_2\sim \frac{N}{n_1^2}}|A(n_1,n_2)|^2=&  \sum_{n_2\sim \frac{N}{n_1^2}}\left|\sum_{d|(n_1,n_2)}\mu(d)A(n_1/d,1)A(1,n_2/d)\right|^2\\
        \leq &N^{\varepsilon} \sum_{n_2\sim \frac{N}{n_1^2}}\sum_{d|(n_1,n_2)}|A(n_1/d,1)|^2|A(1,n_2/d)|^2\\
        \leq & N^{\varepsilon}\sum_{d|n_1}|A(n_1/d,1)|^2\sum_{n_2\sim \frac{N}{n_1^2d}}|A(1,n_2)|^2
        \leq \frac{N^{1+\varepsilon}}{n_1^2}\sum_{d|n_1}\frac{|A(n_1/d,1)|^2}{d}. \qedhere
    \end{align*}
\end{proof}

\subsection{Voronoi summation formula for $\SL(3,\mathbb{Z})$:}
	 Let $F$ be an $\SL(3,\mathbb{Z})$ Maass form with $(m,n)$th Fourier coefficient $A(m,n)$ and let $\tilde{F}$ be its dual form with Fourier coefficients $A(n,m)$. Then we have a summation formula for $A(1,m)$ twisted by additive characters. We precisely follow the expression of Corollary 3.7 of  \cite{Goldfeld2015-qu}, which we summarize in the lemma below.
	\begin{lemma}[Voronoi type summation formula]\label{voronoi}Let $\psi(x)\in \mathbb{C}_c^{\infty}(0,\infty)$ and let $a,\bar{a},q\in \mathbb{Z}$ with $(a,q)=1$. Then
	 	\begin{equation*}\label{key}
			\begin{split}
				\sum_{m=1}^{\infty}A(1,m)e\bigg(\frac{m\bar{a}}{q}\bigg)\psi(m)
				=&\frac{q\pi^{-5/2}}{4i}\sum_{\pm}\sum_{m_1|q}\sum_{m_2>0}\frac{A(m_2,m_1)}{m_1m_2} S(a,\pm m_2;qm_1^{-1})\Psi_{0,1}^{\pm}\bigg(\frac{m_2m_1^2}{q^3}\bigg).
			\end{split}
		\end{equation*}	
	\end{lemma}
	In the above lemma,   $\Psi_{0,1}^{\pm}\big(\frac{m_2m_1^2}{q^3}\big)=\Psi_0\big(\frac{m_2m_1^2}{q^3}\big)\pm\frac{\pi^{-3}q^3}{m_1^2m_2i}\Psi_1\big(\frac{m_2m_1^2}{q^3}\big)$ consists of four terms. But we will only estimate $\Psi_0(x)$ with the help of the following lemma by \cite{Li} (Lemma 6.1) and consider only the term $\Psi_0(x)$. The estimate of $\Psi_1(x)$ is quite similar, so are of the remaining three terms.
	\begin{lemma}\label{Voronoipsi}
		Suppose $\psi(x)$ is a smooth function compactly supported on $[X,2X]$ and  $\Psi_0(x)$ is defined as above, then for any fixed integer $K\geq1$ and $xX\gg1$, we have
		\begin{equation*}\label{key}
			\begin{split}
\Psi_0(x)=&2\pi^4xi\int\limits_{0}^{\infty}\psi(y)\sum_{j=1}^{K}\frac{c_j\cos(6\pi x^{1/3}y^{1/3})+d_j\sin(6\pi x^{1/3}y^{1/3})}{(\pi^3xy)^{j/3}}dy+O\left((xX)^{\frac{-K+2}{3}}\right),
			\end{split}
		\end{equation*}		
		where $c_j$ and $d_j$ are constants depending on $\alpha_i$, in particular, $c_1=0,\ d_1 = \frac{-2}{\sqrt{3\pi}}$.
	\end{lemma}

We also note that the oscillatory part $e(3x^{1/3}y^{1/3})$ is independent of the sum over $j$ and the non-oscillatory terms $(\pi^3xy)^{-j/3}$ decrease with increasing $j$ provided $xy\gg1$. So $\sum_{j=1}^{K}(\pi^3xy)^{-j/3}$ is asymptotic to the term $(\pi^3xy)^{-1/3}$, i.e., $j=1$. Hence, we take $K$ sufficiently large so that the error term $O\big((xX)^{\frac{-K+2}{3}}\big)$ can be dropped from further consideration. For such $K$, we will only consider only the term $j=1$. 

 For $xy\ll 1$, as the term $e(3x^{1/3}y^{1/3})$ is non-oscillatory, it can be absorbed into the smooth function of $\psi(y)$. Thus, we will consider 
\begin{equation}\label{estimateofpsi}
	\Psi_0(x)\asymp \sum_{\pm} 2\pi^4c_{\pm}xi\int\limits_{0}^{\infty} \psi(y)\frac{ e(\pm 3x^{1/3}y^{1/3})}{(\pi^3xy)^{1/3}}\ \ dy +O(T^{-A})
	.\end{equation}
\subsection{Poisson summation formula}
Let $f(x)$ be a compactly supported smooth function and $C(n)$ be a periodic function modulo $q$. Then by the Poisson summation formula, we get 
\begin{align}
    \sum_{n\in \BZ}C(n)f(n)
    =& \frac{1}{q}\sum_{n\in \BZ}\sum_{b\bmod q}C(b)e_q(nb)\int_\bR f(y)e\left(-\frac{ny}{q}\right)dy.\label{Poissonsum}
\end{align}
In particular, if $C(n)=e_q(an)$, we have 
\begin{align}
    \sum_{n\in \BZ}e_q(an)f(n)
    =&\sum_{\substack{n\in \BZ\\n\equiv -a\bmod q}}\int_\bR f(y)e\left(-\frac{ny}{q}\right)dy.\label{Poissonadditive}
\end{align}
\subsection{Stationary Phase Analysis}
To treat the oscillatory integrals of one variable, we will invoke the following lemmas.
     When the phase function does not have a stationary point, we will use the following lemma from \cite{munshi2015Jams}. 
    \begin{lemma}\label{repeatedint}
     Let $g(x)$  be a compactly supported smooth function supported in $[a,b]$ satisfying $g^j(x)\ll_{a,b,j} 1$. Let $f(x)$ be a real valued smooth function satisfying $|f'(x)|\geq \Theta_f$ and $|f^{(j)}(x)|\ll \Theta_f$ for $j\geq 2$. Then for any $j\in \BN$, we have 
     \begin{align}
         \int_{a}^be(f(x))g(x)dx\ll_{a,bj,\varepsilon} \Theta_f^{-j+\varepsilon}. 
     \end{align}
    \end{lemma}
    \begin{remark}
    Frequently in this paper, we will mention that an oscillatory integral is negligibly small by repeated integration by parts if the first derivative of its phase function is bigger than $(N^{\varepsilon})$ for any $\varepsilon>0$ throughout the support of the smooth function $u(x)$. That is actually a direct consequence of this lemma.   
\end{remark}
    When the phase function have a unique stationary point, we will use the following lemma from \cite{Blomer_Khan_Young_2013} by Blomer, Khan, and Young. So, we restate Proposition 8.2 of \cite{Blomer_Khan_Young_2013} below.  
	\begin{lemma}\label{StationaryPhase1}
		Let $0<\delta<1/10, \Theta_g,\Theta_f,\Omega_g,L,\Omega_f>0$ and let $Z:=\Omega_f+\Theta_f+\Theta_g+L+1$ and we also assume that 	
		\begin{equation}\label{8.7}
			\Theta_f\geq Z^{3\delta},\ \ L\geq \Omega_g\geq \frac{\Omega_fZ^{\delta/2}}{\Theta_f^{1/2}}
			.\end{equation}
		Let $g(x)$ be a compactly supported smooth function with support in a length $L$ and satisfying the derivative $g^{(j)}(x)\ll \Theta_g\Omega_g^{-j}$ and let $x_0$ be the unique point such that  $f'(x_0)=0$, where $f(x)$ is a smooth function satisfying 
		\begin{equation}\label{8.8}
			f''(x)\gg \Theta_f\Omega_f^{-2},\ \ f^{(j)}(x)\ll \Theta_f \Omega_f^{-j},\ \ \ \forall j\in \mathbb{N}
			.\end{equation} 
		Then the oscillatory integral $I=\int\limits_{-\infty}^{\infty}g(x)e(f(x))dx$ would have the asymptotic expression (for arbitrary $A>0$)
		\begin{equation}\label{8.9}
			I=\frac{e(f(x_0))}{\sqrt{f''(x_0)}}\sum_{n\leq 3\delta^{-1}A}p_n(x_0)+O_{A,\delta}(Z^{-A}),
		\end{equation}
		where \begin{equation}\label{8.10}
			\begin{split}p_n(x_0)&=\frac{e^{\pi i/4}}{n!}\bigg(\frac{i}{2f''(x_0)}\bigg)^n G^{(2n)}(x_0),\\ \text{ where } G(x)&=g(x)e(f(x)-f(x_0)-f''(x_0)(x-x_0)^2/2).\end{split}
		\end{equation}
		Each $p_n$ is a rational function in derivatives of $f$ satisfying 
		\begin{equation}\label{8.11}
			\frac{d^j}{dx_0^j}p_n(x_0)\ll \Theta_g(\Omega_g^{-j}+\Omega_f^{-j})((\Omega_g^2\Theta_f/\Omega_f^2)^{-n}+\Theta_f^{-n/3})
			.\end{equation}
	\end{lemma}
	\begin{remark}
		As observed in \cite{Blomer_Khan_Young_2013}, from (\ref{8.7}) and (\ref{8.11}), in the asymptotic expression (\ref{8.10}), every term is smaller than the preceding term. So it is enough to consider the leading term in the asymptotic provided we verify \ref{8.7}. 
	\end{remark}

\subsection{$\delta$-method of Duke-Friedlander-Iwaniec}
To separate the oscillations, we will employ the circle method, specifically the $\delta$-method of Duke, Friedlander and Iwaniec (chapter $20$, \cite{Iwaniec2022-qo}) with the conductor lowering mechanism of Munshi. Here the $\delta$-symbol represents the function
	\begin{equation}\label{deltasymboldefn}
		\delta(n):\mathbb{Z}\rightarrow\{0,1\},\ \  \text{such that }\ \ \delta(n)=\begin{cases}
			0 \text{ if } n\neq 0\\
			1\text{ if } n=0
		\end{cases}
		.\end{equation}
	\begin{lemma}\label{delta} \cite{munshi2022Gl3Gl2} 
    We have
		\begin{equation*}
			\delta(n)=\frac{1}{C}\sum_{1\leq q\leq C}\frac{1}{q}\ \sideset{}{^*}\sum_{a\bmod q}e\bigg(\frac{an}{q}\bigg)\int\limits_{-\infty}^{\infty}g(q,x)e\bigg(\frac{nx}{qQ}\bigg)dx,
		\end{equation*} 
		where the sum over $a$ is over the reduced residue class of $q$ (signified by $*$)  and for any $\alpha>1$, $g(q,x)$ satisfies
		\begin{align}
				&	g(q,x)=1+O\bigg(\frac{1}{qC}\bigg(\frac{q}{C}+|x|\bigg)^{\alpha}\bigg),\ \ \ g(q,x)\ll|x|^{-\alpha},\label{gqxshape}\\
                & x^j\frac{\partial ^j}{\partial x^j}g(q,x)\ll \log C\min\left\{\frac{C}{q},\frac{1}{|x|}\right\},\label{gqxderivative}\\
                & \int_{\bR} (|g(q,x)|+|g(q,x)|^2)dx\ll_{\varepsilon} C^{\varepsilon}.\label{gqxintegral}
			\end{align}
	\end{lemma}
    In particular, $g(q,x)$ is supported in $[-N^{\varepsilon},N^{\varepsilon}]$ with negligible error term. When $q\gg C^{1-\varepsilon}$ or  $|x|\gg C^{-\varepsilon}$, by (\ref{gqxderivative}), we have $$x^jg^j(q,x)\ll_{\varepsilon} C^{\varepsilon}.$$ This is the generic scenario. When $q\ll C^{1-\varepsilon}$ and $|x|\ll C^{-\varepsilon}$, by (\ref{gqxshape}), $g(q,x)$ can be taken to be $1$ with a negligible error term. This is the non-generic scenario. 
\subsection{Duality principle}
\begin{lemma}[Duality principle]\label{dualitylemma}
		Let $\phi: \mathbb{Z}^2\rightarrow \BC$. For any sequence of complex numbers $\{a_m\}_{m \in \mathbb{N}}$, we have
        \begin{align*}
			\sum_n\left|\sum_m a_m\phi(m,n)\right|^2\ll \left(\sum_m |a_m|^2\right)\sup_{\|b\|_2=1}\sum_m\left|\sum_n b(n)\phi(m,n)\right|^2,
		\end{align*}
		where the supremum is taken over all sequences of complex numbers $\{b(n)\}_{n \in \mathbb{N}}$ such that $$\|b\|_2=\sqrt{\sum_n|b(n)|^2}=1.$$
	\end{lemma}    
\subsection{Character sum bound}
For the Kloosterman sum $S(a,b,c)$, we use the Weil's bound 
\begin{align}
    S(a,b,c)\ll & c^{1/2+\varepsilon}(a,b,c)^{1/2}\nonumber\\
    \ll& c^{1/2+\varepsilon}\sum_{d|c}d^{1/2}\cdot\delta(d|a)\cdot \delta(d|b),\label{weilsbound}
\end{align}
and for the Ramanujan sum $\fc_c(a)$, we use the following bound 
\begin{align}
    \fc_{c}(a)\ll (a,c)\ll \sum_{d|c}d\cdot \delta(d|a). \label{ramanujanbound}
\end{align}
\section{Delta method and the first Dualization}
\subsection{Separation of oscillations}
Let us recall the definition of $S'$ from \eqref{MainEquation}. We rewrite the expression as 
$$S'=\sum_{h}\sum_{n}A(1,n)\sum_{m}A(1,m)V\left(\frac{n}{N}\right)V\left(\frac{m}{N}\right)V\left(\frac{h}{H}\right)\delta(m-n-h),$$ 
where the delta symbol refers to the Kronecker-delta symbol. Then we appeal to Lemma \ref{delta} and apply the expansion formula for delta symbol. 
  We split the $q$-sum into a dyadic partition of length $Q$.  Finally, we get 
\begin{align}S'\ll_{\varepsilon} C^{\varepsilon} \sup_{Q\ll C}|S(Q)|,\label{S'defn}\end{align}
where 
\begin{align}
    S(Q)=&\int_{-\infty}^{\infty} \frac{1}{C}\sum_{q}\frac{1}{q}V\left(\frac{q}{Q}\right)g(q,x)\ \\
    &\times \sumstar_{a\bmod q} \sum_{h} e\left(\frac{-ah}{q}\right)e\left(\frac{-xh}{qC}\right)V\left(\frac{h}{H}\right)\label{hsum}\\
    &\times \sum_{n}A(1,n)e\left(\frac{-an}{q}\right)e\left(\frac{-xn}{qC}\right)V\left(\frac{n}{N}\right)\label{nsum}\\
     &\times \sum_{m}A(1,m)e\left(\frac{am}{q}\right)e\left(\frac{xm}{qC}\right)V\left(\frac{m}{N}\right)\label{msum}.
\end{align}
To keep it tidy, we will write $S(Q)$ as $S$ by dropping the dependence on $Q$ from writing.  
\subsection{Poisson summation formula ($h$-sum)}
\begin{lemma}
We have,
\begin{align}
    \sumstar_{a\bmod q} \sum_{h} e\left(\frac{-ah}{q}\right)e\left(\frac{-xh}{qC}\right)V\left(\frac{h}{H}\right)\asymp H\sum_{|h|\ll \frac{QN^{\varepsilon}}{H}}\sumstar_{\substack{a\bmod q\\ a\equiv h\bmod q}}W(\cdots)+O(N^{-2025}).\label{hsumafterpoisson}
\end{align}
\end{lemma}
\begin{proof}
We will apply the Poisson summation formula on the $h$-sum in (\ref{hsum}). As $x\ll C^{\varepsilon}$ and $\ h\ll H,$
if we choose $H\ll C$, the exponential term $e\left(\frac{-xh}{qC}\right)$ can be absorbed into the smooth function for all $Q$.  Now  by the Poisson summation formula (\ref{Poissonadditive}), we get 
\begin{align}
    S_{\CH}=\sum_{h}e\left(\frac{-ah}{q}\right)V\left(\frac{h}{H}\right)
    =& H\sum_{\substack{h\in \BZ\\h\equiv a \bmod q}}\int_{-\infty}^{\infty} e(-hHy/q)V\left(y\right)dy .
\end{align}
As $y^jV^j(y)\ll 1$, by Lemma \ref{repeatedint}, the integral is negligibly small unless $h\ll \frac{QN^{\varepsilon}}{H}.$
In that range, the $y$-integral can be absorbed into the smooth functions of $q$ and $h$.   
\end{proof}
\begin{remark}
    The sum over $h\ll \frac{QN^{\varepsilon}}{H}$ actually denotes a  smooth sum $W(h)$ supported in $[-\frac{QN^{\varepsilon}}{H}, \frac{QN^{\varepsilon}}{H}]$ which satisfies $y^jW^j(y)\ll_{j} 1$. However, for brevity we have dropped the weight function from writing.
\end{remark}
\begin{remark}
    Going further, we can entirely replace the $a$ variable with the $h$ variable as $a\equiv h\bmod q $. 
\end{remark}
\subsection{Voronoi summation formula ($m$ and $n$-sum)}
To dualize the $m$-sum
$$S_\CM:=\sum_{m}A(m)e\left(\frac{hm}{q}\right)\psi(m),\ \ \text{where } \psi(y)=e\left(\frac{xm}{qC}\right)V\left(\frac{m}{N}\right),$$
we employ the Voronoi-type summation formula (Lemma \ref{voronoi}). Thus, we immediately get 
\begin{equation}
    S_\CM=\frac{q\pi^{3/2}}{2}\sum_{\pm}\sum_{m_1|q}\sum_{m_2>0}\frac{A(m_2,m_1)}{m_1m_2}S(\overline h,\pm m_2; qm_1^{-1}) \Psi^{\pm}\left(\frac{m_2m_1^2}{q^3}\right).
\end{equation}
Now by using the expression (\ref{estimateofpsi}) and a change of variable in the $y$-integral ($y\mapsto yN$), we have 
\begin{align}\label{PsiAsympMsum}
     \Psi^{\pm}\left(\frac{m_2m_1^2}{q^3}\right)\asymp N^{2/3} \left(\frac{m_2m_1^2}{q^3}\right)^{2/3}\CI_{\CM}, 
\end{align}
where 
\begin{align}
   \CI_{\CM}:= \int_{0}^{\infty}e\left(\frac{xNy}{qC}\pm \frac{3(m_2m_1^2Ny)^{1/3}}{q}\right)V(y) dy .\label{integralmsum}
\end{align}
Thus, after applying \eqref{PsiAsympMsum} and \eqref{integralmsum} we have
\begin{align}
    S_{\CM}\asymp \frac{N^{2/3}}{q}\sum_{\pm}\sum_{m_1|q}\sum_{m_2>0}\frac{A(m_2,m_1)}{(m_1^2m_2)^{1/3}}\cdot m_1\cdot S(\overline h,\pm m_2; qm_1^{-1})\cdot \CI_\CM+O(N^{-2025}).\label{msumaftervoronoi}
\end{align}

Similarly we apply Voronoi summation formula on the $n$-sum and get 
\begin{align}
    S_{\CN}\asymp \frac{N^{2/3}}{q}\sum_{\pm}\sum_{n_1|q}\sum_{n_2>0}\frac{A(n_2,n_1)}{(n_1^2n_2)^{1/3}}\cdot n_1\cdot S(\overline h,\pm n_2; qn_1^{-1})\cdot \CI_\CN+O(N^{-2025}),\label{nsumaftervoronoi}
\end{align}
where 
\begin{align}
   \CI_{\CN}:= \int_{0}^{\infty}e\left(-\frac{xNz}{qC}\pm \frac{3(n_2n_1^2Nz)^{1/3}}{q}\right)V(z) dz .\label{integralnsum}
\end{align}
Here all the choices of $\pm$ are allowed. Finally putting \eqref{hsumafterpoisson}, \eqref{msumaftervoronoi} and \eqref{nsumaftervoronoi} together we have 
\begin{align}
    S\asymp  \frac{H}{CQ}\sum_{q}V\left(\frac{q}{Q}\right) \sum_{|h|\ll \frac{QN^{\varepsilon}}{H}}\cdot \int_{-\infty}^{\infty}g(q,x)S_\CM\cdot S_\CN.\label{Sbeforelinear}
\end{align}

\section{Treating the  integral transforms}
Now, we look at the triple integral in \eqref{Sbeforelinear}:
 \begin{align}
     \mathcal{I}(m,n,q):=&\int_{-\infty}^{\infty}g(q,x) \cdot \CI_\CM\cdot \CI_\CN\cdot  dx\nonumber\\
     =& \int_{x}g(q,x)\int_{y\sim 1}e\left(\frac{xNy}{qC}\pm \frac{3(m_2m_1^2Ny)^{1/3}}{q}\right)\int_{z\sim 1}e\left(-\frac{xNz}{qC}\pm \frac{3(n_2n_1^2Nz)^{1/3}}{q}\right)dzdydx.\label{Itriple integral}
 \end{align}
Let us recall that by (\ref{gqxshape}), $g(q,x)$ is essentially supported in $|x|\leq C^{\varepsilon}$.  We will consider the case for when $x$ is positive because the contribution from negative $x$ is exactly symmetrical to the positive $x$. Further, since the contribution of $S$ for the part of $|x|\leq N^{-2028}$ is at most $O(N^{-2025})$ and thus negligibly small, we will only consider the case for that $x>N^{-2028}$. We split the integral into dyadic intervals of length $X$ where $N^{-2028}<X<C^{\varepsilon}$. So we write 
\begin{align}
    \CI(m,n,q)=&2\sum_{\substack{N^{-2028}\leq X\leq N^{\varepsilon}\\\text{dyadic sum }}
    }\CI(m,n,q,X)+O(N^{-2025})\label{Imnq},\\ \text{ where }\CI(m,n,q,X):=&\int_{x\sim X}g(q,x) \cdot \CI_\CM\cdot \CI_\CN\cdot  dx.\label{I(m,n,q,X)defn}
\end{align}
 Let us note that in the $\CI_\CM$-integral or the $\CI_\CN$-integral, the oscillation related to the $x$-integral: $e(xN/qC)$ is of the size $$\frac{xNy}{qC}\sim \frac{XN}{QC}.$$ 
 Then according to the size of $\frac{xN}{qC}$, we will split our analysis into two cases. When $\frac{XN}{QC}\leq N^{\varepsilon}\iff X\leq \frac{N^{\varepsilon}QC}{N}$, $e(ay)$ can be considered as a flat function with no oscillation and we can treat $\CI_\CM$ and $\CI_\CN$-integrals by repeated integration by parts. But when $X\geq \frac{N^{\varepsilon}QC}{N}$, we will treat the above integral by stationary phase analysis. 
 \subsection{$X\leq \frac{QCN^{\varepsilon}}{N}$}
 When $X\leq \frac{qCN^{\varepsilon}}{N}$, we can absorb the term $e(\frac{xNY}{qC})$ into the smooth function $V(y)$ . By repeated integration by parts (Lemma \ref{repeatedint}), the integral 
 \begin{align*}
   \CI_{\CM}:= \int_{0}^{\infty}e\left(\frac{xNy}{qC}\pm \frac{3(m_2m_1^2Ny)^{1/3}}{q}\right)V(y) dy 
\end{align*}
 is negligibly small unless $(m_1^2m_2)\ll \frac{N^{\varepsilon}Q^3}{N}$.
  In this range, we can treat the $\CI_\CM$-integral as a smooth function. Similarly, $\CI_\CN$-integral is negligibly small unless $(n_1^2n_2)\ll \frac{N^{\varepsilon}Q^3}{N}$ and in that range, $\CI_\CN$ is a smooth function. Hence, for $N^{-2028}\ll X\ll \frac{N^{\varepsilon}QC}{N}$, we have  
  \begin{align}
      |\CI(m,n,q,X)|\ll X\cdot \delta\left((m_1^2m_2)\ll \frac{N^{\varepsilon}Q^3}{N}\right)\cdot \delta\left((n_1^2n_2)\ll \frac{N^{\varepsilon}Q^3}{N}\right).
  \end{align}
Here, we have used (\ref{gqxintegral}) to bound the $L^1$ norm of the $x$-integral. 

At this point, we evaluate $S$ (\ref{msum}). By Lemma \ref{rambound}, $S_\CM$ can be bounded by 
\begin{align*}
    S_{\CM}\ll &\frac{N^{2/3}}{q}\Bigg|\underset{n_1^2n_2\ll q^3N^{\varepsilon}/N}{\sum_{n_1|q}\sum_{n_2>0}}\frac{A(n_2,n_1)}{(n_1^2n_2)^{1/3}}\cdot n_1\cdot S(\overline h,\pm n_2; qn_1^{-1})\cdot \CI_\CN\Bigg|
    \ll \frac{N^{2/3}}{\sqrt{q}}\cdot \frac{q^{3/2}}{\sqrt{N}}\cdot \frac{q^{1/2}}{N^{1/6}}\sim q^{3/2}. 
\end{align*}
Similarly, $S_\CN\ll q^{3/2}$. Hence, 
\begin{align}
    S_{X\leq \frac{QC}{N}}\ll &\sup_{X\ll \frac{QC}{N}}\frac{H}{CQ}\sum_{q\sim Q}\sum_{|h|\ll \frac{QN^{\varepsilon}}{H}}\cdot X\cdot S_\CM\cdot S_\CN
    \ll\frac{Q^5}{N}.\label{SXsmall}
\end{align}
At this stage, for the benefit of further computations we introduce a new notation, $K:=N/C^2$. This implies that $C=\sqrt{N/K}$ and thus we have 
\begin{align}S'_{X\text{ small}}\ll \frac{N^{3/2+\varepsilon}}{K^{5/2}}.\label{SwhenXsmall}\end{align}
 For the rest of the paper, we will assume $X\geq \frac{QCN^{\varepsilon}}{N}$.
 \\
\subsection{$X\geq \frac{QCN^{\varepsilon}}{N}$} When $X\geq \frac{QCN^{\varepsilon}}{N}$, we will now evaluate $I_\CM$ and $I_\CN$. We recall that
$$\CI(m,n,q,x)=\int g(q,x)V(x)\cdot \CI_\CM\cdot \CI_\CN ~dx.$$
 
\begin{lemma}\label{I(m,n,q,X)lemma} We have
   \begin{align}\CI(m,n,q,X)\asymp&\frac{QC}{NX}\cdot V\left(\frac{m_1^2m_2}{X^{3}N^2/C^3}\right)V\left(\frac{n_1^2n_2}{X^{3}N^2/C^3}\right)\nonumber\\&\times  \int_{x}e\left(\frac{2\sqrt{C}(\sqrt{n_1^2n_2}-\sqrt{m_1^2m_2})}{q\sqrt{x}}\right) 
    \times g(q,x)V\left(\frac{x}{X}\right) dx +O(N^{-2025}).\label{I(m,n,q,X)eqn}\end{align}
\end{lemma}
\begin{proof}
We start with the expression of $\CI(m,n,q,X)$ derived from (\ref{Itriple integral}) and (\ref{I(m,n,q,X)defn}). If we write 
 $a=\frac{xN}{qC} \text{ and } b=\left(\frac{m_2m_1^2N}{q^3}\right)^{1/3},$
 then the phase function $f(y)$ of the oscillatory integral is $f(y)=ay\pm 3by^{1/3}$. 
As $X$ is taken to be positive and as $a\geq N^{\varepsilon}$ implies that $a+by^{-2/3}\geq N^{\varepsilon}$, the `+' case would give negligible contribution by repeated integration by parts (Lemma \ref{repeatedint}). Hence, we will only consider `-' case. In that scenario, by the stationary phase analysis (Lemma \ref{StationaryPhase1}) we get
\begin{align}
     \CI_\CM
     \asymp &\sqrt{\frac{QC}{XN}}e\left(-\frac{2\sqrt{m_1^2m_2C}}{q\sqrt{x}}\right)V\left(\frac{m_1^2m_2}{X^{3}N^2/C^3}\right)+O(N^{-2025}).
\end{align}
We similarly evaluate $\CI_\CN$, where the $-$ case will be negligibly small by Lemma \ref{repeatedint} and the $+$ case will have a stationary point. Hence, through a similar evaluation, we get 
\begin{align*}
    \CI_\CN\asymp &\sqrt{\frac{QC}{XN}}e\left(\frac{2\sqrt{n_1^2n_2C}}{q\sqrt{x}}\right)V\left(\frac{n_1^2n_2}{X^{3}N^2/C^3}\right)+O(N^{-2025}).
\end{align*}
This concludes the proof of the lemma. 
\end{proof}
\subsection{The $x$-integral}\label{subsec.xint}
We start with recalling the consideration $ X\geq \frac{QCN^{\varepsilon}}{N}$. Let us denote the dyadic size of $m_1^2m_2$ and $n_1^2n_2$ by $\mathbf{M_0(X):=X^{3}N^2/C^3}.$ In the $x$-integral of (\ref{I(m,n,q,X)eqn}), i.e.,
\begin{align}
   \CI_{\CX}:= \int_{x}e\left(\frac{2\sqrt{C}(\sqrt{n_1^2n_2}-\sqrt{m_1^2m_2})}{q\sqrt{x}}\right) 
    \times g(q,x)V\left(\frac{x}{X}\right) dx,\label{I_cX}
\end{align}
our aim will be to transform the structure of the phase function such that it would be a linear function of $n_1^2n_2$ and $m_1^2m_2$. 

\begin{lemma}\label{linearizationlemma}
    Let $Y_0:=\frac{C^2}{QX^2N}$. Let $s=c_1+it_1$ and $z=c_2+it_2$ for some $c_1,c_2>1$. Let $V(x)$ be some smooth function supported in $[1,2]$ satisfying $V^j(x)\ll N^{j\varepsilon}$ and $V_0(x/X):=g(q,x)V(x/X)$. Then 
    \begin{align}
    \CI_\CX\asymp &\frac{1}{2\pi^2}\int\limits_{|t_1|\ll N^{\varepsilon}}\int\limits_{|t_2|\ll N^{\varepsilon}} \hat V_0(s)\frac{\Gamma(-2s+2+z)\Gamma(-z)}{\Gamma(-2s+2)} \nonumber\\
    & \times \frac{X}{Y_0} \int e(y(n_1^2n_2-m_1^2m_2))V\left(\frac{y}{Y_0}\right)dy~~~ V(\cdots ) ~~ds dz+O(N^{-A}).\label{Linearizationmaineqn}
\end{align}
\end{lemma}
\begin{proof}
We note that 
$$\frac{\partial^j}{\partial x^j}g(q,x)V(x/X)\ll _jX^{-j}N^{j\varepsilon}.$$
Thus, we will redefine the smooth weight function of (\ref{I_cX}) as $$V_0(x/X):=g(q,x)V(x/X),$$
where $V_0(x)$ is supported in $[1,2]$ and satisfies $V_0^j(x)\ll N^{j\varepsilon}$. So we start with 
\begin{align}
    \CI_{\CX}:= \int_{x}e\left(\frac{2\sqrt{C}(\sqrt{n_1^2n_2}-\sqrt{m_1^2m_2})}{q\sqrt{x}}\right) 
    V_0\left(\frac{x}{X}\right) dx,
\end{align}
for any smooth $V_0(x)$, supported in $[1,2]$ and satisfies $V_0^j(x)\ll N^{j\varepsilon}$.
Then we can change the variable $x$ to $y$, where 
$$x=\frac{Y^2X}{y^2}, ~~~\text{where}~~~ \mathbf{Y:=\frac{2\sqrt{C}}{q\sqrt{X}(\sqrt{n_1^2n_2}+\sqrt{m_1^2m_2})}}, ~~dx=-2Y^2Xy^{-3}dy.$$
 With this change of variable, the $\CI_\CX$-integral transforms into 
\begin{align}
    \CI_\CX=-2Y^2X\int_{y}e(y(n_1^2n_2-m_1^2m_2))\cdot y^{-3}\cdot V_0\left(\frac{Y^2}{y^2}\right)dy.\label{IXfirstafterlinear}
\end{align}
Though the new phase function is linear in $m_1^2m_2$ and $n_1^2n_2$, we have $(\sqrt{n_1^2n_2}+\sqrt{m_1^2m_2})$ present inside the term $Y$, and we cannot separate them directly to apply A.M.-G.M. inequality.

Now, we recall the dyadic ranges of $q, m_1^2m_2, n_1^2n_2$ and $x$
$$q\sim Q,~~m_1^2m_2, n_1^2n_2\sim M_0(X)=\frac{X^3N^2}{C^3},~~x\sim X.$$
Thus, 
$$Y=\frac{2\sqrt{C}}{q\sqrt{X}(\sqrt{n_1^2n_2}+\sqrt{m_1^2m_2})}\sim \frac{\sqrt{C}C^{3/2}}{Q\sqrt{X}X^{3/2}N}\sim \frac{C^2}{QX^2N}\sim Y_0, ~~\text{where}~~\mathbf{Y_0:=\frac{C^2}{QX^2N}}.$$
So, we have $c_1<\frac{Y}{Y_0}<c_2$ for some absolute constants $c_1,c_2>0$. As the smooth function $V_0(x)$ was supported in $[1,2]$, we also get 
$y\sim Y\sim Y_0$. So, we can artificially introduce a compactly supported smooth function $V(x)$ (identically $1$ in the range $x\in [c_1,c_2]$ and satisfying $v^{(j)}(x)\ll N^{j\varepsilon}$) into the integral in the form $V(y/Y_0)$:
\begin{align}
\CI_\CX=   -2X\int \frac{1}{y^3}e(y(n_1^2n_2-m_1^2m_2))\times Y^2V_0\left(\frac{Y^2}{y^2}\right)\times  V\left(\frac{y}{Y_0}\right)dy+O(N^{-A}).\label{IXsecondafterlinear}
 \end{align}

 But term of the form $(\sqrt{n_1^2n_2}+\sqrt{m_1^2m_2})$ is present inside the smooth function $V(Y^2/y^2)$ in (\ref{IXfirstafterlinear}). Hence, we would apply the Mellin-Barnes representation (Lemma \ref{MellinBarnesLemma}) to separate them. 
 In the notation of Lemma \ref{MellinBarnesLemma}, let $A=\sqrt{n_1^2n_2}, \ B=\sqrt{m_1^2m_2},\  \CX=A+B$. Then we have
$$Y^2V_0\left(\frac{Y^2}{y^2}\right)=\frac{4C}{q^2X(A+B)^2}V_0\left((A+B)^{-2}\cdot \frac{4C}{q^2Xy^2}\right)=\frac{4C}{q^2X}\cdot \CX^{-2}V_0(\CX^{-2}Z),$$
where $Z=\frac{4C}{q^2Xy^2}$. 
Then by Lemma \ref{MellinBarnesLemma}, we get 
\begin{align*}
    Y^2V_0\left(\frac{Y^2}{y^2}\right)=&\frac{4C}{q^2X}\cdot \CX^{-2}V_0(\CX^{-2}Z)\\
    =& \frac{-1}{4\pi^2}\int\limits_{|t_1|\ll N^{\varepsilon}}\int\limits_{|t_2|\ll N^{\varepsilon}} \hat V_0(s)\frac{\Gamma(-2s+2+z)\Gamma(-z)}{\Gamma(-2s+2)}\nonumber\\
    &\times\left(\frac{4C}{q^2X}\right)^{1-s}Y_0^{2s}M_0^{z/2+s-1-z/2} \cdot \frac{y^{2s}}{Y_0^{2s}}\cdot \frac{(n_1^2n_2)^{z/2}}{M_0^{z/2}}\cdot \frac{(m_1^2m_2)^{s-1-z/2}}{M_0^{s-1-z/2}}dsdz+O_A(N^{-A}).
\end{align*}
Here $s=c_1+it_1$ and $z=c_2+it_2$ for some suitably chosen $c_1 $ and $c_2$. We recall that
$$q\sim Q, ~~m_1^2m_2,n_1^2n_2\sim M_0,~~y\sim Y_0.$$As $t_1\ll N^{\varepsilon}$ and $t_2\ll N^{\varepsilon}$, the following terms can be absorbed into their respective smooth functions: $\left(\frac{q}{Q}\right)^{2s},\ \left(\frac{y}{Y_0}\right)^{2s},\ \left(\frac{n_1^2n_2}{M_0}\right)^{z/2}$ and $\left(\frac{m_1^2m_2}{M_0}\right)^{s-1-z/2}$.  
We denote them by $V(\cdots)$ and write 
\begin{align}
    Y^2V_0\left(\frac{Y^2}{y^2}\right)&=\frac{-1}{4\pi^2}\int\limits_{|t_1|\ll N^{\varepsilon}}\int\limits_{|t_2|\ll N^{\varepsilon}} \hat V_0(s)\frac{\Gamma(-2s+2+z)\Gamma(-z)}{\Gamma(-2s+2)}\left(\frac{4C}{Q^2X}\right)^{1-s} Y_0^{2s}(M_0)^{s-1}\cdot V(\cdots )dsdz\nn\\
    &= Y_0^2   \frac{-1}{4\pi^2}\int\limits_{|t_1|\ll N^{\varepsilon}}\int\limits_{|t_2|\ll N^{\varepsilon}} \hat V_0(s)\frac{\Gamma(-2s+2+z)\Gamma(-z)}{\Gamma(-2s+2)} \cdot  V(\cdots )ds dz.\nn
\end{align}
Here we have used the fact that $Y_0=\frac{2\sqrt{C}}{Q\sqrt{X}M_0}$. Once we put this expression back into (\ref{IXsecondafterlinear}) and absorb $\frac{Y_0^3}{y^3}$ into the smooth function $V(y/Y_0)$, we get our required expression of $\CI_\CX$, 
concluding the proof of the lemma.
\end{proof}
Putting (\ref{Linearizationmaineqn})into (\ref{Imnq}) and putting it back into (\ref{Sbeforelinear}), we get 
\begin{align}
    S\ll &\frac{HN^{\varepsilon}}{CQ}\sum_{\pm}\sum_{\substack{\frac{QCN^{\varepsilon}}{N}\leq X\leq N^{\varepsilon}\\\text{dyadic }} } \frac{N^{4/3}}{Q^2M_0^{2/3}}\cdot  \frac{Q^2X^2}{C} \int\limits_{|t_1|\ll N^{\varepsilon}}\int\limits_{|t_2|\ll N^{\varepsilon}} \hat V_0(s)\frac{\Gamma(-2s+2+z)\Gamma(-z)}{\Gamma(-2s+2)}\nonumber \\& \int_{y\sim Y_0}\sum_{q\sim Q} \sumstar_{h\ll \frac{QN^{\varepsilon}}{H}} \times {\sum_{m_1|q}\sum_{m_2>0}}A(m_2,m_1)m_1\cdot S(\overline h,\pm m_2; qm_1^{-1})e(m_1^2m_2y)V\left(\frac{m_1^2m_2}{M_0}\right)\nonumber\\
    &\times {\sum_{n_1|q}\sum_{n_2>0}}\overline A(n_2,n_1)m_1\cdot S(\overline h,\pm n_2; qn_1^{-1})e(-n_1^2n_2y)V\left(\frac{n_1^2n_2}{M_0}\right)V(\cdots )dy  ds dz.
\end{align}
We note that the double integral over $s$ and $z$ is bounded by $N^{\epsilon}$, i.e.,
$$ \int\limits_{|t_1|\ll N^{\varepsilon}}\int\limits_{|t_2|\ll N^{\varepsilon}} \hat V_0(s)\frac{\Gamma(-2s+2+z)\Gamma(-z)}{\Gamma(-2s+2)}dsdz\ll N^{\varepsilon}.$$
Though the smooth functions of other variables might depend on $s$ and $z$, in reality, our smooth functions $V(x)$ are arbitrary, only satisfying the decay condition $V^j(x)\ll N^{j\varepsilon}$. Thus, we can replace the above double integral by its upper bound. Then using AM-GM inequality, we can bound $S$ by 
\begin{align}
    &\sup_{\frac{QCN^{\varepsilon}}{N}\leq X\leq N^{\varepsilon}} \frac{HN^{4/3+\varepsilon}X^2}{C^2QM_0^{2/3}}\cdot\int_{y\sim Y_0}\sum_{q\sim Q}\sumstar_{h\ll \frac{QN^{\varepsilon}}{H}} \Bigg|\underset{m_1^2m_2\sim M_0}{\sum_{m_1|q}m_1\sum_{m_2>0}}A(m_2,m_1)\cdot S(\overline h, m_2; qm_1^{-1})e(m_1^2m_2y)\Bigg|^2.\nonumber\\
\end{align}
Here, we have only considered the $+$ case as the other cases follow similar analysis and lead to the same bound. 
As $m_1$ runs over the divisors of $q$, i.e., $O(Q^{\varepsilon})$-many elements, by Cauchy's inequality, we can write 
$\big|\sum_{m_1|q}F(m_1)\big|^2\leq C^{\varepsilon}\sum_{m_1|q}|F(m_1)|^2 $. 
Then we make a change of variable $q\mapsto qm_1$ in order to exchange the order of the $q$ and the $m_1$-sum and take a smooth dyadic partition of the $h$-sum. Finally, we get 
\begin{lemma}
\begin{align}
    S\ll_{\varepsilon} &\sup_{\frac{QCN^{\varepsilon}}{N}\leq X\leq N^{\varepsilon}}\sup_{H'\ll\frac{QN^{\varepsilon}}{H}}  \frac{1}{M_0}\cdot\frac{HN^{2+\varepsilon}X^3}{C^3Q}\sum_{m_1\ll Q}m_1^2\nonumber\\
    &\times \int_{y\sim Y_0}\sum_{q\sim Q/m_1}\sumstar_{\substack{h\sim H\\(h,q)=1}} \Bigg|\sum_{m_2\sim M_0/m_1^2}A(m_2,m_1)\cdot S(\overline h, m_2; q)e(m_1^2m_2y)\Bigg|^2.\label{S.before.large.sieve}
\end{align}
\end{lemma}
\section{ Duality principle and the second dualization}
In order to exchange the order of the summations, similar to \cite{ALM}, we apply the duality principle  (Lemma \ref{dualitylemma}) in (\ref{S.before.large.sieve}).  Then using Lemma \ref{anotherramanujanbound} we get the following upper bound of $S$
\begin{align}
     S\ll &\sup_{\frac{QCN^{\varepsilon}}{N}\leq X\leq N^{\varepsilon}}\sup_{H'\ll\frac{QN^{\varepsilon}}{H}}\cdot \frac{HN^{2+\varepsilon}X^3}{C^3Q}
     \cdot \sum_{m_1\leq Q}\sum_{d|m_1}\frac{|A(m_1/d,1)|^2}{d}\cdot \Delta,\label{Sjustafterlargesieve}
\end{align}
where the dual norm $\Delta$ is defined by
\begin{equation}
    \Delta:=\sup_{||\alpha(y,h,q)||_2=1}\sum_{m_2\sim\frac{M_0}{m_1^2}}\Bigg|\int_{y\sim Y_0}\sum_{q\sim Q/m_1} \ \sumstar_{h\sim H'}\alpha(y,q,h)S(\overline h, m_2; q)e(m_1^2m_2y)\Bigg|^2.
\end{equation}
Now we want to dualize the $m_2$-sum with the Poisson summation formula. So we open up the absolute square of $\Delta$ and get 
\begin{align}
    \Delta=\sup_{||\alpha||_2=1}\int\limits_{y_1\sim Y_0}\sum_{q_1\sim\frac{Q}{m_1}} \sum_{h_1\sim H'}\alpha(y_1,q_1,h_1)\int\limits_{y_2\sim Y_0}\sum_{q_2\sim\frac{Q}{m_1}} \sum_{h_2\sim H'}\overline \alpha(y_2, q_2,h_2) \cdot T_m,\label{Delta.before.Poisson}
\end{align}
where 
\begin{align}
    T_m:=\sum_{m_2}S(\overline h_1, m_2; q_1)S(\overline h_2, m_2; q_2) e(m_1^2m_2(y_1-y_2))V\left(\frac{m_2}{M_0/m_1^2}\right). \label{S_m}
\end{align}
\subsection{Poisson summation formula ($m_2$-sum)}
We can observe that the phase function of the exponential term in (\ref{S_m}) is linear in $m_2$. This is the direct consequence of the liniearization of the phase function in Lemma \ref{linearizationlemma}. As a result, the Poisson summation on the $m_2$-sum would dissolve any remaining analytic oscillation and would provide us with a restriction in terms of the $m$-sum (dual variable) and the integral over $y_1$ and $y_2$. So we apply the Poisson summation formula on the $m_2$-sum,  to derive the following lemma.\ref{linearizationlemma}
\begin{lemma}
    Let $T_m$ be as defined in (\ref{S_m}). Then we have 
    \begin{align}
    T_m\asymp \frac{M_0}{m_1^2}\sum_{m\ll Q^2Y_0N^{\varepsilon}}\mathcal{C}\cdot W\left(\frac{Q^2(y_1-y_2)-m}{Q^2/M_0}\right)+ O_A(N^{-A}), 
\end{align}
where 
\begin{align}
    \mathcal{C}=\frac{1}{q_1q_2}\sum_{\beta \bmod q_1q_2}S(\overline h_1, \beta;q_1)S(\overline h_2, \beta; q_2) e\left(\frac{\beta m}{q_1q_2}\right), \label{charsum1}
\end{align}
and $W(x)$ is some smooth function supported in $[-N^{\varepsilon}, N^{\varepsilon}]$. 
\end{lemma}
\begin{proof}
We will apply the Poisson summation formula (\ref{Poissonsum}) on the $m_2$-sum (\ref{S_m}). We note that $S(\overline h_1, \beta;q_1)S(\overline h_2, \beta; q_2)$ is periodic modulo $q_1q_2$. Hence, by invoking the Poisson summation formula (\ref{Poissonsum}) (we denote the dual variable by $m$), we get 
\begin{align}
    T_m=\frac{M_0}{m_1^2q_1q_2}\sum_{m\in \mathbb{Z}}&\sum_{\beta\bmod q_1q_2}S(\overline h_1, \beta;q_1)S(\overline h_2, \beta; q_2)e\left(\frac{\beta m}{q_1q_2}\right)\nonumber\\&\int_{\bR}e(M_0y(y_1-y_2))e\left(\frac{-mM_0y}{m_1^2q_1q_2}\right)V\left(y\right)dy.\nonumber
\end{align}
We treat the $y$-integral by repeated integration by parts. By Lemma \ref{repeatedint}, the $y$-integral is negligibly small unless 
\begin{align}
    \left|M_0(y_1-y_2)-\frac{mM_0}{m_1^2q_1q_2}\right|\ll N^{\varepsilon}\iff \left|Q^2(y_1-y_2)-m\right|\ll \frac{N^{\varepsilon}Q^2}{M_0}.
\end{align}
In that range, the $y$-integral can be absorbed into the smooth functions. So, we have 
\begin{align}
    \int_{\bR}e(M_0y(y_1-y_2))e\left(\frac{-mM_0y}{m_1^2q_1q_2}\right)V\left(y\right)dy=W\left(\frac{Q^2(y_1-y_2)-m}{Q^2/M_0}\right)+O(N^{-A}),
\end{align}
where $W(x)$ is supported in $[-N^{\varepsilon}, N^{\varepsilon}]$. 
So $W\left(\frac{Q^2(y_1-y_2)-m}{Q^2/M_0}\right)$ also implies an weaker restriction $m\ll Q^2Y_0N^{\varepsilon}$. 
\end{proof}
\begin{remark}
    We will not use $W(\cdots)$ as a smooth function on $m, y_1$ or $y_2$. It will be used to give an effective measure of the support of the variables $m, y_1$ and $y_2$.  
\end{remark}
At this point, we will separate our analysis for `zero frequency' ($m=0$) and `non-zero frequency' ($m\neq 0$).

\subsection{Zero frequency}
When $m=0$, opening up each Kloosterman sums in $\mathcal{C}$ and executing the $\beta$-sum, we get the condition $q_1=q_2$ and the remaining sums simplifies into a Ramanujan sum $\mathfrak{c}_{q_1}(h_1-h_2)$. On the other hand the smooth function $W\left(\frac{Q^2(y_1-y_2)-m}{Q^2/M_0}\right)$ give a restriction in terms of $y_1$ and $y_2$-integral. In a nutshell, we have
\begin{align}
    \mathcal{C}=&\frac{1}{q_1q_2}\sum_{\beta \bmod q_1q_2}S(\overline h_1, \beta;q_1)S(\overline h_2, \beta; q_2)=\delta(q_1=q_2)\cdot \mathfrak{c}_{q_1}(h_1-h_2),
\end{align}
and 
\begin{align}T_{m=0}\asymp \frac{M_0}{m_1^2}\delta(q_1=q_2)\cdot \mathfrak{c}_{q_1}(h_1-h_2)W\left(\frac{(y_1-y_2)}{M_0^{-1}}\right).\label{Smzerofreq}\end{align}
$\mathfrak c_q(x)$, the Ramanujan sum modulo $q$ is bounded by $\mathfrak c_q(x)\ll (x,q)$ (see \cite[(3.5)]{Iwaniec2022-qo}). So for a fixed $h_1$, we write
\begin{align}\label{zerofreqchar}
   &\sum_{h_2\sim H'}\mathfrak{c}_{q_1}(h_1-h_2) \ll \sum_{h_2\sim H'}(q_1,h_1-h_2)
   \ll N^{\varepsilon}(H'+q_1).
\end{align}
We recall $H'\ll\frac{QN^{\varepsilon}}{H}$ and $q_1\sim\frac{Q}{m_1}$. Now for $m=0$, using (\ref{Smzerofreq}) and  (\ref{zerofreqchar}), we  bound $\Delta$ in (\ref{Delta.before.Poisson}) by  
\begin{align*}
    \Delta_{m=0}
    \ll_{\varepsilon} & \frac{M_0}{m_1^2}\sup_{||\alpha||_2=1}\Bigg| \int\limits_{y_1\sim Y_0}\sum_{q_1\sim\frac{Q}{m_1}} \sumstar_{h_1\ll \frac{QN^{\varepsilon}}{H}}|\alpha(y_1,q_1,h_1)|^2\int\limits_{y_2}W(M_0(y_1-y_2))\sumstar_{h_2\ll \frac{QN^{\varepsilon}}{H}}|\mathfrak{c}_{q_1}(h_1-h_2)|\Bigg|\\
    \ll_{\varepsilon}& \frac{1}{m_1^2}\cdot \left(\frac{Q}{m_1}+\frac{N^{\varepsilon}Q}{H}\right). 
\end{align*}
Here, we have used the AM-GM inequality to get $|\alpha_1\alpha_2|\ll |\alpha_1|^2+|\alpha_2|^2$ and only considered the $|\alpha_1|^2$ as the other term follows the same analysis. For brevity, we will often use the notation $\alpha_i$ to represent $\alpha(y_i,q_i,h_i)$. Finally, putting it into (\ref{Sjustafterlargesieve}), we get 
\begin{align}
    S_{m=0}\ll_{\varepsilon} & \sup_{\frac{QCN^{\varepsilon}}{N}\leq X\leq N^{\varepsilon}}\sup_{H'\ll\frac{QN^{\varepsilon}}{H}}\cdot \frac{HN^{2+\varepsilon}X^3}{C^3Q}
     \cdot \sum_{m_1\leq Q}\sum_{d|m_1}\frac{|A(m_1/d,1)|^2}{d}\cdot \Delta\nn\\
     \ll_{\varepsilon}& \sup_{\frac{QCN^{\varepsilon}}{N}\leq X\leq N^{\varepsilon}}\frac{HN^{2+\varepsilon}X^3}{C^3Q}\cdot Q\ll \frac{HN^{2+\varepsilon}}{C^3}\asymp HN^{1/2}K^{3/2}. \label{Sdiagonal}
\end{align}
\begin{remark}
    This bound is smaller than $HN$ as long as $$HN^{1/2}K^{3/2}\ll HN\iff K\ll N^{1/3}.$$
\end{remark}
\subsection{Non-zero frequency}\label{nonzerofreq}
In this section, we will simplify the character sum  (\ref{charsum1})
$$ \mathcal{C}=\frac{1}{q_1q_2}\sum_{\beta \bmod q_1q_2}S(\overline h_1, \beta;q_1)S(\overline h_2, \beta; q_2) e\left(\frac{\beta m}{q_1q_2}\right),$$
for the non-zero frequencies,  i.e., $m\neq 0$ in a method similar to \cite{pal2023secondmomentdegreelfunctions}. But before that we will define a few notations to handle the coprimality issues arising from all the variables that are not pairwise coprime. This would help us to apply the Chinese Remainder Theorem in the required places.  

\begin{notation}\label{firstnotation}
To analyze the character sum $\mathcal{C}$ for $m\neq 0$,we will define a new set of notations. Let $d_q:=(q_1,q_2)$, $v_1:=(q_1/d_q,d_q^{\infty})$ and $v_2:=(q_2/d_q,d_q^{\infty})$, where,
$$(a,b^{\infty}):=\prod_{p|b}\{p^{\nu}\ |\ p^{\nu}|a\text{ and  }p^{\nu+1}\nmid a\}. $$  
So $v_1$ and $v_2$ extracts all the prime power factors from $q_1/d_q$ and $q_2/d_q$ corresponding to the prime factors of $d_q$. Thus, if we denote $u_1:=\frac{q_1}{v_1d_q}$ and $u_2:=\frac{q_2}{v_2 d_q}$, we get 
$$q_1=d_qv_1u_1,\ q_2=d_qv_2u_2,\ (v_1,v_2)=1,\ (u_1,u_2)=1,\ (u_i, d_qv_1v_2)=1. $$

Let $d_q=d_0d_1d_2$ where $(d_0,v_1v_2)=1,\ d_1=(d_q,v_1^\infty),\ d_2=(d_q,v_2^\infty)$. We also observe that $d_1$ and $d_2$ are coprime to $\tilde m=\frac{m}{d_q}$, but $d_0$ may not be. Let $d_m=(\tilde m,d_0^{\infty})$ and $\Fm=\frac{m}{d_qd_m}$. Thus, we write 
$$mh_i=(\Fm h_i)\cdot (d_0d_m)\cdot d_1\cdot d_2.$$ 

So we can split $q_1$ and $q_2$ as $q_i=d_qv_iu_i$ and rewrite the sum over $q_1$ and $q_2$ as
$$\sum_{q_1\sim\frac{Q}{m_1}}\sum_{q_2\sim \frac{Q}{m_1}}\cdots\mapsto \sum_{d_q\leq Q}\underset{(v_1,v_2)=1}{\sum_{v_1|d_q^{\infty}}\sum_{v_2|d_q^{\infty}}}\sum_{u_1\sim \frac{Q}{m_1d_qv_1}}\sum_{u_2\sim \frac{Q}{m_1d_qv_2}}\cdots .$$
\end{notation}
Then as in \cite{pal2023secondmomentdegreelfunctions}, we deduce the following lemma.
\begin{lemma}
We have
\begin{align}
    \CC=\sumstar_{x_0\bmod d_0}~\sumstar_{x_1\bmod d_1}~\sumstar_{x_2\bmod d_2}e\left(\frac{\overline{u_1}u_2A_1}{mh_1}\right)e\left(\frac{\overline{u_2}u_1A_2}{mh_2}\right)U(\cdots),
\end{align}
where, 
\begin{align}\label{defA1A2}
    A_1\equiv \begin{cases}
        v_2\overline{v_1}&\pmod {\Fm h_1}\\
        v_2\overline{v_1}\overline{[1+\bar x_0d_m]}&\pmod {d_0d_m}\\
        v_2\cdot  \overline{\bar x_1+v_1}& \pmod {d_1}\\
        \overline{v_1x_2}&\pmod {d_2}
    \end{cases}~~~~~~\text{and}~~~~~~~
     A_2\equiv \begin{cases}
        v_1\overline{v_2}&\pmod {\Fm h_2}\\
        v_1\overline{v_2}[1+\bar x_0d_m]&\pmod {d_0d_m}\\
        \overline{v_2x_1}&\pmod {d_1}\\
        v_1 \cdot \overline{\bar x_2+v_2}& \pmod {d_2}
    \end{cases}.
\end{align}
\end{lemma}
\begin{proof}
    See Appendix \ref{charsumappen}. 
\end{proof}
\begin{remark}
Without loss of generality, we may assume $v_1\leq v_2$. For the other case, we can just reverse the role of $q_1$ and $q_2$. 
\end{remark}

We recall that
\begin{align}
    \Delta(q,h)=&\sup_{||\alpha||_2=1}\int\limits_{y_1\sim Y_0}\sum_{q_1\sim\frac{Q}{m_1}} \sum_{h_1\sim H'}\alpha(y_1,q_1,h_1)\int\limits_{y_2\sim Y_0}\sum_{q_2\sim\frac{Q}{m_1}} \sum_{h_2\sim H'}\overline \alpha(y_2, q_2,h_2) \cdot T_m\nn\\
    \text{where, }T_m& \asymp \frac{M_0}{m_1^2}\sum_{m\ll Q^2Y_0}\mathcal{C}\cdot W\left(\frac{Q^2(y_1-y_2)-m}{Q^2/M_0}\right)+ O_A(N^{-A}).\label{Deltaaftermpoisson}
\end{align}
Here we have $(d_q,h_1h_2)=1$, $(u_1, h_1)=1$ and $(u_2,h_2)=1$. We also have $(v_1v_2, h_1h_2)=1$ and $(m, u_1u_2)=1$.We further want $(u_2,h_1)=1$, which is not necessarily true. So we enforce that by extracting out the g.c.d. of $u_2$ and $h_1$ as shown below.
\begin{notation}\label{secondnotation}
    Let $d_{h}=(u_2,h_1)$. Then by abuse of notation, we write $h_1$ as $d_h h_1$ and $u_2$ as $d_h u_2$. The new $h_1$ and $u_2$ are co-prime to each other and it infers the following:
$$\sum_{h_1\sim H'}\sum_{u_2\sim \frac{Q}{m_1d_qv_2d_h}}\cdots(h_1, u_2)\rightsquigarrow \sum_{d_h\leq H'}\underset{(h_1,u_2)=1}{\sum_{h_1\sim \frac{H'}{d_h}}\sum_{u_2\sim \frac{Q}{m_1d_qv_2d_h}}}\cdots(d_hh_1, d_hu_2).$$
We also emphasize on the fact that with the new notation of $h_1$ and $u_2$, we have $(mh_1h_2, u_2)=1$. Let $\hat h_2=(h_2, (m h_1)^{\infty})$ and $\eta_2=\frac{h_2}{\hat h_2}$.

For brevity, by the notation $\smc$ we denote
\begin{align}
    \smc:=&\sum_{d_q\leq Q}\sum_{d_h\leq H'}\underset{(v_1,v_2)=1}{\sum_{v_1|d_q^{\infty}}\sum_{v_2|d_q^{\infty}}} \int\limits_{y_1\sim Y_0}\int\limits_{y_2\sim Y_0}\sum_{\tilde m} W\left(\frac{Q^2(y_1-y_2)-\tilde md_q}{Q^2/M_0}\right)\nonumber\\
    & \sumstar_{x_0\bmod d_0}~\sumstar_{x_1\bmod d_1}~\sumstar_{x_2\bmod d_2}~\sum_{h_1\sim \frac{H'}{d_h}}~\sum_{h_2\sim H'}.\label{sumshorthand}
\end{align}
\end{notation}

\section{Analysis of the non-zero frequencies ($m\neq 0$)}\label{Sectionq_1nonzero}

Following the notations defined in Notation \ref{firstnotation}, Notation \ref{secondnotation} particularly the notation $\smc$ (\ref{sumshorthand}), we can rewrite $\Delta(q,h)$ (\ref{Deltaaftermpoisson}) as
\begin{align*}
    \Delta(q,h)=&\sup_{||\alpha||_2=1}\frac{M_0}{m_1^2}\smc\sum_{u_1\sim \frac{Q}{m_1d_qv_1}}\alpha(y_1,d_qv_1u_1,d_hh_1)\\
    & \sum_{u_2\sim \frac{Q}{m_1d_qv_2d_h}}\overline \alpha(y_2,d_qv_2u_2d_h,h_2)e\left(\frac{\overline{u_1}u_2A_1}{mh_1}\right)e\left(\frac{\overline{d_hu_2}u_1A_2}{mh_2}\right).\\
\end{align*}

Now, our aim is to apply the Poisson summation formula on the $u_1$-sum. In order to do so, we get rid of $\alpha(y_1,q_1,h_1)$ by applying the Cauchy's inequality. Hence, we get 
\begin{align}\Delta(q,h)\ll \sup_{||\alpha||_2=1}\frac{M_0}{m_1^2}\cdot S_0^{1/2}\cdot S_1^{1/2},\label{DeltaOneCauchyMain}\end{align}
where 
\begin{align}
S_0&=\smc \hat h_2\sum_{u_1\sim \frac{Q}{m_1d_qv_1}}|\alpha(y_1,d_qv_1u_1,h_1)|^2, \label{S_0defn}
\end{align}
and 
\begin{align}
S_1=& \smc \frac{1}{\hat h_2} \sum_{u_1\sim \frac{Q}{m_1d_qv_1}}\Bigg|\sum_{u_2\sim \frac{Q}{m_1d_qv_2d_h}}\overline \alpha(y_2,d_qv_2u_2d_h,h_2)e\left(\frac{\overline{u_1}u_2A_1}{mh_1}\right)e\left(\frac{\overline{d_hu_2}u_1A_2}{mh_2}\right)\Bigg|^2.\label{S_1}
\end{align}
\begin{remark}
We note that we have multiplied $S_0$ by $\hat h_2$ and divided $S_1$ by $\hat h_2$. This normalization will help us at the later stage. For clarity one may assume $\hat h_2=1$ which is the generic situation. 
\end{remark}
Firstly we derive an upper bound of $S_0$ (\ref{S_0defn}) in the following lemma. 
\begin{lemma}
    Let $S_0$ be as defined in (\ref{S_0defn}). Then we have
    \begin{align}
        S_0\ll_{\varepsilon} \frac{Q^2Y_0H'N^{\varepsilon}}{M_0}.\label{S_0bound}
    \end{align}
\end{lemma}
\begin{proof}
    If we recall the definition of $\smc$ from (\ref{sumshorthand}), we get 
    \begin{align*}
        S_0=&\sum_{d_q\leq Q}\sum_{d_h\leq H'}\underset{(v_1,v_2)=1}{\sum_{v_1|d_q^{\infty}}\sum_{v_2|d_q^{\infty}}} \int\limits_{y_1\sim Y_0}\sum_{u_1\sim \frac{Q}{m_1d_qv_1}}\sum_{h_1\sim \frac{H'}{d_h}}|\alpha(y_1,d_qv_1u_1,d_hh_1)|^2\nonumber\\&\times\sum_{h_2\sim H'}\sum_{\tilde m\ll \frac{Q^2Y_0}{d_q}}\hat h_2\sumstar_{x_0\bmod d_0}\sumstar_{x_1\bmod d_1}\sumstar_{x_2\bmod d_2} \int\limits_{y_2\sim Y_0}W\left(\frac{Q^2(y_1-y_2)-\tilde md_q}{Q^2/M_0}\right)\\
        \ll& \sum_{d_q\leq Q}\underset{(v_1,v_2)=1}{\sum_{v_1|d_q^{\infty}}\sum_{v_2|d_q^{\infty}}} \int\limits_{y_1\sim Y_0}\sum_{u_1\sim \frac{Q}{m_1d_qv_1}}\sum_{d_h\leq H'}\sum_{h_1\sim \frac{H'}{d_h}}|\alpha(y_1,d_qv_1u_1,d_hh_1)|^2\times\sum_{h_2\sim H'}\sum_{\tilde m\ll \frac{Q^2Y_0}{d_q}}\hat h_2d_q\cdot \frac{1}{M_0}\\
     \ll& \frac{Q^2Y_0H'N^{\varepsilon}}{M_0}.
    \end{align*}
    Here, we have used the fact that 
    \begin{align}
        \sum_{d_q\leq Q}\underset{(v_1,v_2)=1}{\sum_{v_1|d_q^{\infty}}\sum_{v_2|d_q^{\infty}}} \int\limits_{y_1\sim Y_0}\sum_{u_1\sim \frac{Q}{m_1d_qv_1}}\sum_{d_h\leq H'}\sum_{h_1\sim \frac{H'}{d_h}}|\alpha(y_1,d_qv_1u_1,d_hh_1)|^2\ll N^{\varepsilon},\label{alphanorm1}
    \end{align}
    and 
    $$\sum _{h_2\sim H'}\hat h_2\ll H'N^{\varepsilon}. $$
\end{proof}
We open up the absolute square in $S_1$ (\ref{S_1}) and get
\begin{align}
    S_1=&\sum\cdots \sum \frac{1}{\hat h_2} \sum_{u_2\sim \frac{Q}{m_1d_qv_2d_h}}\overline \alpha(y_2,d_qv_2u_2d_h,h_2)\sum_{u_3\sim \frac{Q}{m_1d_qv_2d_h}}\alpha(y_2,d_qv_2u_3d_h,h_2)\cdot S_{q}\label{S_1seconddef},\\
    \text{where}~~S_q=&\sum_{u_1\sim \frac{Q}{m_1d_qv_1}}e\left(\frac{\overline{u_1}(u_2-u_3)A_1}{mh_1}\right)e\left(\frac{(\overline{u_2}-\overline{u_3})\overline{d_h}A_2u_1}{mh_2}\right).\label{S_qdefn}
\end{align}

We will split the analysis of $S_1$ into three cases: 1. Before the application of the Poisson summation fomrula on $u_1$-sum, we treat the case $u_2=u_3$. Then we apply Poisson summation formula on the $u_1$ sum and treat the case 2. $q_1^*=0$ and 3. $q_1^*\neq 0$, where $q_1^*$ is the dual variable. The first two cases are relatively easier and will be handled in this section. We will treat the third case ($q_1^*\neq 0$) separately in Section \ref{Sectionq_1nonzero} and Section \ref{sectioninfcauchy}. 

\subsection{Diagonal term $u_2=u_3$}

We start with the case $u_2=u_3$, when there is no oscillatory factor and we have $S_q=\sum_{u_1\sim \frac{Q}{m_1d_qv_1}}1$. 
\begin{lemma}\label{u_2=u_3}
    Let $S_1$ be as defined in (\ref{S_1seconddef}) and let $u_2=u_3$. Then we have
    \begin{align}
        S_{1, u_2=u_3}\ll_{\varepsilon} \frac{Q^3Y_0H'N^\varepsilon}{M_0} .\label{S_1boundu_2=u_3}
    \end{align}
\end{lemma}
\begin{proof}
    If $u_2=u_3$, we note that $S_q=\sum_{u_1\sim \frac{Q}{m_1d_qv_1}}$ and 
    $$S_1=\smc  \sum_{u_2\sim \frac{Q}{m_1d_qv_2d_h}}|\overline \alpha(y_2,d_qv_2u_2d_h,h_2)|^2\sum_{u_1\sim \frac{Q}{m_1d_qv_1}}1.$$
    Comparing it to $S_0$ (\ref{S_0defn}), we derive 
    $$S_1\ll S_0\cdot Q\ll \frac{Q^3Y_0H'N^\varepsilon}{M_0}.$$
\end{proof}
\subsection{Poisson summation formula on $u_1$-sum}
Now, proceeding with the assumption $u_2\neq u_3$, we apply the Poisson summation formula (\ref{Poissonsum}) on $S_q$ (\ref{S_qdefn}) and denote the dual variable by $q_1^*$. Thus, we get
\begin{align*}
    S_q=&\frac{1}{mh_1h_2}\sum_{q_1^*\in \mathbb{Z}}\sum_{\beta \bmod mh_1h_2}e\left(\frac{\overline{\beta}(u_2-u_3)A_1}{mh_1}\right)e\left(\frac{(\overline{u_2}-\overline{u_3})\barr{d_h}A_2\beta}{mh_2}\right)\\&\times e\left(\frac{\beta q_1^*}{mh_1h_2}\right)\int\limits_{y\sim \frac{Q}{m_1d_qv_1}}e\left(\frac{-q_1^*y}{mh_1h_2}\right)dy.
\end{align*}
Once we change the variable $y\mapsto \frac{Qy}{m_1v_1d_q}$, by repeated integration by parts (Lemma \ref{repeatedint}), we observe that the integral is negligibly small unless 
$$q_1^*\ll \frac{m_1v_1d_qmh_1h_2N^{\varepsilon}}{Q}.$$
So we have 
\begin{align}
    S_q\asymp \frac{Q}{m_1d_qv_1mh_1h_2}\sum_{q_1^*\ll \frac{mh_1h_2m_1d_qv_1N^{\varepsilon}}{Q}}\mathcal{C}'+O(N^{-A}),\label{Sqbeforecharnormalization} 
\end{align}
where 
\begin{align}
    \CC'=\underset{(\beta, mh_1)=1}{\sum_{\beta \bmod mh_1h_2}}e\left(\frac{\overline{\beta}(u_2-u_3)A_1}{mh_1}\right)e\left(\frac{(\overline{u_2}-\overline{u_3})\barr{d_h}A_2\beta}{mh_2}\right)e\left(\frac{\beta q_1^*}{mh_1h_2}\right).\label{charsumstart}
\end{align}
We separate the analysis for $q_1^*=0$ and $q_1^*\neq 0$. 
We note that $m,h_1,h_2$ are not mutually co-prime. So we perform a maneuver similar to the $q_1, q_2$ scenario done at the beginning of Section \ref{nonzerofreq}. 
\begin{notation}\label{notsplitmh1h2}
    We recall that $m=d_q\cdot d_m\cdot \fm$ where $d_m=(\frac{m}{d_q}, d_q^{\infty})$ and $(\fm, d_qd_m)=1$. We denote $\fd=d_q\cdot d_m$. We already had the condition $(\fd, h_1h_2)=1$. 

Now let $\hat h_1=(h_1, (\fm h_2)^{\infty})$ and $\hat h_2=(h_2, (\fm h_1)^{\infty})$. Then let $\mu=\fm\hat h_1\hat h_2$, $\eta_1=\frac{h_1}{\hat h_1}$, $\eta_2=\frac{h_2}{\hat h_2}$. In that case, $\fd$, $\mu$, $\eta_1$ and $\eta_2$ are mutually co-prime. In the generic case, one may think $\fd=1$, $\mu= m$, $\eta_1=h_1$ and $\eta_2=h_2$. 

We also mention that only the notations of $\eta_2$ and $\hat h_2$ which are defined here, are to be used after Lemma \ref{lemmas_1q=0}. Rest of the notations are only for Lemma \ref{lemmas_1q=0}.
\end{notation}

\subsection{Zero-freuqency ($q_1^*=0$)}
\begin{lemma}\label{lemmas_1q=0}
For $q_1^*=0$, we get
    \begin{align}
        S_{1,q_1^*=0}\ll_{\varepsilon} & \frac{N^{\varepsilon}Q^3Y_0^{1/2}}{m_1^2H'M_0}.\label{S_1boundq=0}
    \end{align}
\end{lemma}
\begin{proof}
With the notations defined in Notation \ref{notsplitmh1h2}, we split the character sum $\CC'$ as 
\begin{align}
     \CC'=&\sumstar_{\beta_d\bmod \fd}e\left(\frac{\overline{\beta_d}(u_2-u_3)A_1\barr{\eta_1\mu}\hat h_2}{\fd}\right)e\left(\frac{(\overline{u_2}-\overline{u_3})\barr{d_h}A_2\beta_0 \barr{\eta_2\mu }\hat h_1}{\fd}\right)\nonumber\\
     &\sumstar_{\beta _0\bmod \mu}e\left(\frac{\overline{\beta_0}(u_2-u_3)A_1\barr{\eta_1\fd}\hat h_2}{\mu}\right)e\left(\frac{(\overline{u_2}-\overline{u_3})\barr{d_h}A_2\beta_0 \barr{\eta_2\fd}\hat h_1}{\mu}\right)\nonumber\\
    &\times \sumstar_{\beta_1\bmod \eta_1}e\left(\frac{\overline{\beta_1}(u_2-u_3)A_1\barr{\mu\fd}\hat h_2}{\eta_1}\right)
    \sum_{\beta_2\bmod \eta_2}e\left(\frac{(\overline{u_2}-\overline{u_3})\barr{d_h}A_2\beta_2\barr{\mu\fd} \hat h_1}{\eta_2}\right).\nonumber\\
    =& S((u_2-u_3)A_1\barr{\eta_1\mu}\hat h_2,(\overline{u_2}-\overline{u_3})\barr{d_h}A_2\barr{\eta_2\mu}\hat h_1;\fd)\times S((u_2-u_3)A_1\barr{\eta_1\fd}\hat h_2,(\overline{u_2}-\overline{u_3})A_2\barr{\eta_2\fd}\hat h_1;\mu)\nonumber\\
    &\times \fc_{\eta_1}((u_2-u_3)A_1\barr{\mu\fd}\hat h_2)\times \eta_2\cdot \delta((\overline{u_2}-\overline{u_3})\barr{d_h}A_2\barr{\mu\fd} \hat h_1\equiv 0 \bmod \eta_2).\label{q0charsumfirst}
\end{align}
For the Kloosterman sum modulo $\fd$, we will use the trivial bound 
$$ S((u_2-u_3)A_1\barr{\eta_1\mu}\hat h_2,(\overline{u_2}-\overline{u_3})\barr{d_h}A_2\barr{\eta_2\mu}\hat h_1;\fd)\ll \fd.$$
Using (\ref{weilsbound}) and (\ref{ramanujanbound}) on (\ref{q0charsumfirst}), we get 
\begin{align}
    \CC'\ll &\fd\cdot \eta_2\cdot \mu^{1/2+\varepsilon}\cdot \sum_{d_\mu|\mu}d_\mu^{1/2}\cdot \delta(d_\mu|(u_2-u_3)A_1\barr{\eta_1}\hat h_2)\cdot \delta(d_\mu|(\overline{u_2}-\overline{u_3})\barr{d_h}A_2\barr{\eta_2}\hat h_1)\nonumber\\
    &\times \sum_{d_\eta|\eta_1}d_{\eta}\cdot \delta(d_\eta|(u_2-u_3)A_1\barr{\mu}\hat h_2)\cdot \delta(\eta_2|(\overline{u_2}-\overline{u_3})\barr{d_h}A_2\barr \mu \hat h_1). \label{q0charsumsecond}
\end{align}
Let us recall the definition of $A_1$ and $A_2$ from \eqref{defA1A2}. We note that $(A_1A_2, \mu \eta_1\eta_2)=1$. We  also recall from \ref{notsplitmh1h2} that $\mu=\fm\hat h_1\hat h_2, \eta_1, \eta_2$ are mutually co-prime and $(u_2u_3,\mu\eta_1\eta_2)=1$. With these information we write 
\begin{align}
    \CC'\ll &\fd\cdot \eta_2\cdot\mu^{1/2+\varepsilon}\cdot  \hat h_1^{1/2}\hat h_2^{1/2}\cdot \sum_{d_\mu|\fm}d_\mu^{1/2}\sum_{d_\eta|\eta_1}d_{\eta}\cdot\delta(d_\eta d_{\mu}\eta_2|(u_2-u_3)).
\end{align}
Then for a fixed $u_2$, we get 
\begin{align}
    \sum_{0<|u_3-u_2|\ll \frac{Q}{m_1 v_1d_qd_h}}\CC'
    \ll& N^{\varepsilon}\fd\cdot \sqrt{\fm}\hat h_1\hat h_2\cdot \frac{Q}{m_1v_1d_qd_h}.  
\end{align}
Here we have only considered the case $u_2\neq u_3$ as the case $u_2=u_3$ has already been treated in Lemma \ref{u_2=u_3}. Then after the application of the AM-GM inequality we can bound $S_1$ for $q_1=0$ by 
\begin{align*}
    S_{1,q=0}\ll_{\varepsilon} &\sum\cdots \sum \frac{1}{\hat h_2} \sum_{u_2\sim \frac{Q}{m_1d_qv_2d_h}}|\alpha(y_2,d_qv_2u_2d_h,h_2)|^2\cdot \frac{Q}{m_1d_qv_1mh_1h_2}\sum_{u_3\sim \frac{Q}{m_1d_qv_2d_h}}|\CC'|\\
    \ll_{\varepsilon}& N^{\varepsilon}\sum_{d_q\leq Q}\underset{(v_1,v_2)=1}{\sum_{v_1|d_q^{\infty}}\sum_{v_2|d_q^{\infty}}}\int\limits_{y_2\sim Y_0}\sum_{h_2\sim H'}\sum_{d_h\leq H'}\sum_{u_2\sim \frac{Q}{m_1d_qv_2d_h}} |\alpha(y_2,d_qv_2u_2d_h,h_2)|^2\sum_{\tilde m\ll \frac{Q^2Y_0}{d_q}} \\
    & \times  \int\limits_{y_1\sim Y_0}W\left(\frac{Q^2(y_1-y_2)-\tilde md_q}{Q^2/M_0}\right)~\sum_{h_1\sim \frac{H'}{d_h}}\frac{Q^2\sqrt{d_m}\hat h_1}{m_1^2v_1v_2\sqrt{d_q}\sqrt{m}d_hh_1h_2}\\
    \ll_{\varepsilon}&  N^{\varepsilon}\sum_{d_q\leq Q}d_q^{-3/2}\underset{(v_1,v_2)=1}{\sum_{v_1|d_q^{\infty}}\frac{1}{v_1}\sum_{v_2|d_q^{\infty}}}\frac{1}{v_2}\int\limits_{y_2\sim Y_0}\sum_{h_2\sim H'}\sum_{d_h\leq H'}\sum_{u_2\sim \frac{Q}{m_1d_qv_2d_h}}|\alpha(y_2,d_qv_2u_2d_h,h_2)|^2\cdot \frac{Q^2}{m_1^2}\cdot \frac{QY_0^{1/2}}{H'M_0}\\
    \ll& \frac{N^{\varepsilon}Q^3Y_0^{1/2}}{m_1^2H'M_0}.
\end{align*}
\end{proof}
\section{Cauchy's inequality ad infinitum}\label{sectioninfcauchy}
Now, we will deal with the case when $q_1^*\neq 0$ and $u_2\neq u_3$.  We recall that $h_2=\hat h_2\eta_2$, where $\hat h_2=(h_2, (\fm h_1)^{\infty})$. Thus, we have $(u_1, mh_1\hat h_2)=1$ and $(mh_1\hat h_2, \eta_2)=1$. So for $q_1^*\neq 0$, using (\ref{charsumstart}) we can simplify the character sum $\mathcal{C}'$ as
\begin{align}
    \CC'=&\underset{(\beta, mh_1)=1}{\sum_{\beta \bmod mh_1h_2}}e\left(\frac{\overline{\beta}(u_2-u_3)A_1}{mh_1}\right)e\left(\frac{(\overline{u_2}-\overline{u_3})\barr{d_h}A_2\beta}{mh_2}\right)e\left(\frac{\beta q_1^*}{mh_1h_2}\right)\nn\\
    =& S((u_2-u_3)A_1\hat h_2, (\overline{u_2}-\overline{u_3})\barr{d_h\eta_2}A_2h_1+q_1^*\barr{\eta_2}; mh_1\hat h_2)\nn\\& \cdot \eta_2\cdot \delta((\overline{u_2}-\overline{u_3})\barr{d_h}A_2h_1+q_1^*\equiv 0 \bmod \eta_2).\label{C'maindef}
\end{align}
Then we can rewrite $S_1$ (\ref{S_1seconddef}) as 
$$S_1=\sum\cdots \sum \frac{1}{\hat h_2} \sum_{u_2\sim \frac{Q}{m_1d_qv_2d_h}}\overline \alpha_2\sum_{u_3\sim \frac{Q}{m_1d_qv_2d_h}}\alpha_3\cdot  \frac{Q}{m_1d_qv_1mh_1h_2}\sum_{q_1^*\ll \frac{mh_1h_2m_1d_qv_1N^{\varepsilon}}{Q}}\mathcal{C}'. $$
\begin{remark}
At this point $S_1$ can be bounded by $(Q^2Y_0)^{3/2}\cdot \frac{1}{M_0}\cdot H'^{3/2}\cdot Q.$ Then by (\ref{DeltaOneCauchyMain}) and (\ref{S_0bound}), we derive
$$\Delta\ll_{\varepsilon} (Q^2Y_0)^{5/4}(H')^{5/4}Q^{1/2}\asymp \frac{Q^3C^{5/2}}{X^{5/2}N^{5/4}H^{5/4}}.$$
Here, we have used $Y_0=\frac{C^2}{QX^2N}$ and $H'=\frac{Q}{H}$. Then by (\ref{Sjustafterlargesieve}), we get 
$$S'_{q_1^*\neq 0}\ll_{\varepsilon} \sup_{Q\ll C}\sup_{\frac{QCN^{\varepsilon}}{N}\leq X\leq N^{\varepsilon}}\frac{HN^2X^3}{C^3Q}\cdot \frac{Q^3C^{5/2}}{X^{5/2}N^{5/4}H^{5/4}} \asymp \frac{N^{3/2}}{K^{3/4}H^{1/4}}.$$
Gathering up the bound of $S'$ from other cases (for which refer to (\ref{S'Finalcontribution}) and only consider the dominating terms), we get  
\begin{align}
    S'\ll_{\varepsilon} N^{\varepsilon}\Big(\frac{N^{3/2}}{K^{5/2}}+\frac{N^{5/4}}{K^{1/4}}+HN^{7/8}K^{3/8}+\frac{N^{3/2}}{K^{3/4}H^{1/4}}\Big)
\end{align}
Optimally choosing $K=\frac{N^{5/9}}{H^{10/9}}$, we can bound $S'$ by 
\begin{align}
    S'\ll_{\varepsilon} N^{\varepsilon}\Big(N^{2/18}H^{25/9}+N^{10/9}H^{5/18}+N^{13/12}H^{7/12}\Big).
\end{align}
This bound in non-trivial for $H\geq N^{1/5+\varepsilon}$ which is already an improvement over the result of \cite{dasgupta2024secondmomentgl3standard}.
\end{remark}
To obtain a finer estimate of $S_1$, we want to apply the process of Cauchy's inequality ad infinitum introduced in \cite{ALM}, which is an iterative process of applying Cauchy's inequality followed by the application of the Poisson summation formula. We will demonstrate the first two steps for clarity before going to the $j$th step for arbitrary $j>2$ and evaluate the upper bound of $S_1$ at that point. 

\subsection{First step}
Now we normalize the character sum $\CC'$ (\ref{C'maindef}) as

 \begin{align}\mathcal{C}_1:=\frac{1}{\eta_2\sqrt{mh_1\hat h_2}}\CC'=
& \CC_{1,m}\cdot \CC_{1,h2},
\end{align}
    where
    \begin{align}
        \CC_{1,m}=&\frac{1}{\sqrt{m h_1 \hat h_2}} S((u_2-u_3)A_1\hat h_2, (\overline{u_2}-\overline{u_3})\barr{d_h\eta_2}A_2h_1+q_1^*\barr{\eta_2}; mh_1\hat h_2),\label{C1m}\\
        \CC_{1,h2}=& \delta(\overline{u_2}-\overline{u_3}\equiv -\overline{h_1A_2}d_hq_1^*\bmod \eta_2).\label{C1h2}
    \end{align}
    Let us note that the normalization of $\CC_1$ would help us to get $\CC_1\ll 1$ on average. We will follow this normalization every time we apply the Poisson summation formula. Also note that $\CC_1$ is of modulus $mh_1\hat  h_2\eta_2=mh_1h_2$. 
Then we can rewrite (\ref{S_1seconddef}) as 
\begin{align*}
    S_1=\sum\cdots \sum \cdot \frac{Qh_2}{m_1d_qv_1mh_1h_2}\sum_{0\neq q_1^*\ll \frac{mh_1h_2m_1d_qv_1}{Q}} { \sum_{u_2\sim \frac{Q}{m_1d_qv_2d_h}}\overline \alpha_2\sum_{u_3\sim \frac{Q}{m_1d_qv_2d_h}}\alpha_3} \cdot \frac{\sqrt{mh_1}}{\hat h_2^{3/2}}\mathcal{C}_1(q_1^*, u_2,u_3) .
\end{align*}

Now, we will apply Cauchy's inequality on all but the $u_3$-sum to get
\begin{align*}
    S_1\ll S_{1,1}^{1/2}\cdot S_{2}^{1/2},
\end{align*}
where 
\begin{align}
    S_{1,1}\ll & \sum\cdots \sum \cdot \frac{Qh_2^2}{m_1d_qv_1mh_1h_2}\cdot mh_1\sum_{0\neq q_1^*\ll \frac{mh_1h_2m_1d_qv_1}{Q}} \sum_{u_2\sim \frac{Q}{m_1d_qv_2d_h}}|\alpha_2|^2\nn\\
    \ll&\frac{Q^2Y_0H'}{M_0}\cdot Q^2Y_0\cdot H'^3\asymp \frac{Q^4Y_0^2H'^4}{M_0},\label{S1,1bound}
\end{align}
and 
\begin{align*}
    S_2=\sum \cdots \sum  \cdot \frac{Q}{m_1d_qv_1{mh_1h_2}}\sum_{0\neq q_1^*\ll \frac{mh_1h_2m_1d_qv_1}{Q}} \cdot {\sum_{u_3\sim \frac{Q}{m_1d_qv_2d_h}}\alpha_3\sum_{u_4\sim \frac{Q}{m_1d_qv_2d_h}}\alpha_4}\cdot \frac{1}{\hat h_2^3} S_{2,q}, 
\end{align*}
where 
\begin{align*}
    S_{2,q}={\sum_{u_2\sim \frac{Q}{m_1d_qv_2d_h}}}\CC_1(q_1^*, u_2,u_3)\bar \CC_1(q_1^*, u_2,u_4).
\end{align*}
We recall that $(u_2, mh_1h_2)=1$. We note that $\CC_1(\cdots )\barr \CC_1(\cdots)$ is of modulus $mh_1h_2$. Once, we apply the Poisson summation formula on $S_{2,q}$ and evaluate the integral transform by repeated integration by parts, we get 
\begin{align}
    S_{2,q}=\frac{Q}{m_1d_qv_2md_hh_1h_2}\sum_{q_2^*\ll \frac{md_hh_1h_2m_1d_qv_2N^\varepsilon}{Q}} \sqrt{mh_1\hat h_2}\cdot \CC_2(q_1^*, q_2^*, q_3,q_4),
\end{align}
where 
\begin{align}
    \CC_2(q_1^*, q_2^*, u_3,u_4)=\frac{1}{\sqrt{mh_1\hat h_2}}~\sum_{\substack{\beta_2\bmod mh_1h_2\\(\beta_2,mh_1\hat h_2)=1}}\CC_1(q_1^*, \beta_2,u_3 )\bar \CC_1(q_1^*, \beta_2,u_4 )e_{m h_1\hat h_2\eta_2}(q_2^*\beta_2).
\end{align}
 We can split the character sum $\CC_2$ into $\CC_{2,m}$, and $\CC_{2,h2}$ as in (\ref{C1m}), and (\ref{C1h2}). We analyze them one by one. First we observe that 
\begin{align}
   \CC_{2,m}=&\frac{1}{\sqrt{m h_1\hat h_2}}\sumstar_{\beta_{2}\bmod mh_1\hat h_2 } \CC_{1,m}(q_1^*, \beta_2, u_3 )\bar \CC_{1,m}(q_1^*, \beta_2, u_4)e_{mh_1\hat h_2}(q_2^*\beta_2 \barr{\eta_2})\nonumber\\
   =&\frac{1}{(mh_1\hat h_2)^{3/2}}\sumstar_{\beta_{2}\bmod mh_1\hat h_2 }S((\beta_2-u_3)A_1\hat h_2, ((\overline{\beta_2}-\overline{u_3})\barr{ d_h\eta_2}A_2h_1+q_1^*\barr{\eta_2}); mh_1\hat h_2)\nonumber\\
   &\times S((\beta_2-u_4)A_1\hat h_2, ((\overline{\beta_2}-\overline{u_4})\barr{d_h \eta_2}A_2 h_1+q_1^*\barr{\eta_2}); mh_1\hat h_2)\times  e_{m h_1\hat h_2}(q_2^*\beta_2 \barr{\eta_2}).\label{C2m}
\end{align}
Note that this is exactly the character sum of Lemma B.1 of \cite{ALM}, but normalized so that the average value is 1. We also get 

\begin{align}
    \CC_{2,h2}=& \sumstar_{\kappa_2\bmod \eta_2}\delta(\overline{\kappa_2}-\overline{u_3}\equiv -\overline{h_1A_2}d_hq_1^*\bmod \eta_2)\delta(\overline{\kappa_2}-\overline{u_4}\equiv -\overline{h_1A_2}d_hq_1^*\bmod \eta_2)e_{\eta_2}(\kappa_2 q_2^*\barr{mh_1\hat h_2})\nn\\
    =& \delta(u_3\equiv u_4\bmod \eta_2)\times e_{\eta_2}(q_2^*\barr{mh_1\hat h_2}\cdot \barr{\bar u_3-\overline{h_1A_2}d_hq_1^*}).\label{C2h2}
\end{align}
We can ignore the last exponential term as it vanishes after the next application of Cauchy's inequality  where we only keep the sum over $u_4$ inside. If we denote $B=\frac{md_hh_1h_2m_1d_qv_2}{Q}$ (we note that $B\geq 1$), we get 
\begin{align}
    S_2=\sum \cdots \sum \cdot \frac{v_2d_h}{Bv_1}\sum_{0\neq q_1^*\ll \frac{BN^{\varepsilon}v_1}{d_hv_2}} \cdot {\sum_{u_3\sim \frac{Q}{m_1d_qv_2d_h}}\alpha_3\sum_{u_4\sim \frac{Q}{m_1d_qv_2d_h}}\alpha_4}\cdot \frac{1}{B}\sum_{q_2^*\ll BN^{\varepsilon}}\frac{\sqrt{mh_1}}{\hat h_2^{5/2}}\CC_2. 
\end{align}

\subsection{Second Step}
We recall that $S_1\ll S_{1,1}^{1,2}S_2^{1,2}$. Now we apply Cauchy's inequality in $S_2$ keeping all but the $u_4$-sum outside and get 
$$S_1\ll S_{1,1}^{1/2}S_{2,1}^{1/2^2}S_{3}^{1/2^2},$$
where 
\begin{align}
    S_{2,1}\ll & \sum \cdots \sum \cdot \sum_{u_3\sim \frac{Q}{m_1d_qv_2d_h}}|\alpha_3|^2\cdot \frac{v_2d_h}{Bv_1}\sum_{0\neq q_1^*\ll \frac{BN^{\varepsilon}v_1}{d_hv_2}} \frac{1}{B}\sum_{q_2^*\ll BN^{\varepsilon}}\cdot mh_1\nn\\
    \ll& \frac{Q^2Y_0H'}{M_0}\cdot Q^2Y_0 H'\asymp \frac{Q^4Y_0^2H'^2}{M_0},\label{S2,1finalbound}
\end{align}
and 
\begin{align}
    S_{3}=\sum\cdots \sum \frac{v_2d_h}{Bv_1}\sum_{0\neq q_1^*\ll \frac{BN^{\varepsilon}v_1}{d_hv_2}} \cdot \frac{1}{B}\sum_{q_2^*\ll BN^{\varepsilon}}\cdot  \sum_{u_3\sim \frac{Q}{m_1d_qv_2d_h}}\frac{1}{\hat h_2^5}\left|\sum_{u_4\sim \frac{Q}{m_1d_qv_2d_h}}\alpha_4\CC_2(\cdots, u_3, u_4)\right|^2.
\end{align}
In (\ref{S2,1finalbound}), we have used that 
\begin{align}
  \frac{Q}{m_1d_qv_2md_hh_1h_2}\sum_{q_2^*\ll \frac{md_hh_1h_2m_1d_qv_2N^\varepsilon}{Q}} =  \frac{1}{B}\sum_{q_2^*\ll BN^{\varepsilon}}\ll \left(N^{\varepsilon}+\frac{1}{B}\right)\ll N^{\varepsilon},\label{q_2*bound}
\end{align}
since $\frac{Bv_1}{d_hv_2}\gg N^{-\varepsilon}$ (otherwise there would be no non-zero $q_1^*$) and $v_1\leq v_2$, which further implies that $\frac{1}{B}\ll \frac{N^{\varepsilon}v_1}{d_hv_2}\ll N^{\varepsilon}$. 
\subsection{$j$th step}
Then after $j$-many similar iterated applications of Cauchy's inequality we get
$$S_1\ll \prod_{i=1}^{j-1}S_{i,1}^{1/2^i}\cdot S_{j}^{1/2^{j}},$$
where 
\begin{align}
S_{j-1,1}=    \sum \cdots \sum \cdot\frac{v_2d_h}{Bv_1}\sum_{0\neq q_1^*\ll \frac{BN^{\varepsilon}v_1}{d_hv_2}} \frac{1}{B}\sum_{q_2^*\ll BN^{\varepsilon}}\cdots \frac{1}{B}\sum_{q_{j-1}^*\ll BN^{\varepsilon}}\cdot \sum_{u_j\sim \frac{Q}{m_1d_qv_2d_h}}|\alpha_j|^2\cdot \sqrt{mh_1},\label{S_j,1bound}
\end{align}
and 
\begin{align}
    S_{j}=&\sum\cdots \sum \cdot \frac{v_2d_h}{Bv_1}\sum_{0\neq q_1^*\ll \frac{BN^{\varepsilon}v_1}{d_hv_2}} \cdot\cdots \frac{1}{B}\sum_{q_{j-1}^*\ll BN^{\varepsilon}} {\sum_{u_{j+1}\sim \frac{Q}{m_1d_qv_2d_h}}\alpha_{j+1}\sum_{u_{j+2}\sim \frac{Q}{m_1d_qv_2d_h}}\bar \alpha_{j+2}}\nn\\
    &\times \frac{1}{\hat h_2^{2^{j-1}+1}}\sum_{u_j\sim \frac{Q}{m_1d_qv_2d_h}}\cdot \cC_{j-1}(\cdots, u_j, u_{j+1})\bar \cC_{j-1}(\cdots, u_j,u_{j+2} ).\label{S_jstart}
\end{align}
Following the calculation in (\ref{S2,1finalbound}), we get 
\begin{align}
    S_{j,1}\ll \frac{Q^4Y_0^2H'^2}{M_0}, ~~~\text{for }j\geq 2.\label{S_j,1finalbound}
\end{align}
Once we apply the Poisson summation formula on the $u_j$-sum in $S_j$ (\ref{S_jstart}), and normalize the character sum, we get 
\begin{align}
    S_{j}=&\sum\cdots \sum \cdot \frac{v_2d_h}{Bv_1}\sum_{0\neq q_1^*\ll \frac{BN^{\varepsilon}v_1}{d_hv_2}}  \cdots \frac{1}{B}\sum_{q_{j-1}^*\ll BN^{\varepsilon}}{\sum_{u_{j+1}\sim \frac{Q}{m_1d_qv_2d_h}}\alpha_{j+1}\sum_{u_{j+2}\sim \frac{Q}{m_1d_qv_2d_h}}\alpha_{j+2}}\nonumber\\&\cdot \frac{\sqrt{mh_1}}{\hat h_2^{2^{j-1}+1/2}}\cdot \frac{1}{B}\sum_{q_j^*\ll BN^{\varepsilon}}\cdot \CC_j.\label{Sjafterpoisson}
   \end{align}
Here we have
\begin{align}
    \CC_j=&\frac{1}{\sqrt{mh_1\hat h_2}}\sumstar_{\beta_j\bmod mh_1h_2}\CC_{j-1}(\cdots , \beta_j, u_{j+1})\bar \CC_{j-1}(\cdots, \beta_j, u_{j+2})e_{mh_1h_2}(q_j^*\beta_j)
    = \CC_{j,m}\cdot  \CC_{j,h2},
\end{align}
where 
\begin{align}
    \CC_{j,m}=&\frac{1}{\sqrt{m h_1\hat h_2}}\sumstar_{\beta_j\bmod m h_1\hat h_2}\CC_{j-1,m}(\cdots u_{j+1})\bar \CC_{j-1,m}(\cdots u_{j+2})e_{mh_1\hat h_2}(q_j^*\beta_j\barr{\eta_2}),\label{Cjm}
\end{align}
and 
\begin{align}
    \CC_{j,h2}=\delta(u_{j+1}\equiv u_{j+2}\bmod \eta_2).\label{Cjh2}
\end{align}

The character sums of the form $\CC_{j,m}$ (\ref{Cjm}) (without the normalizing factor of $\frac{1}{\sqrt{mh_1\hat h_2}}$ at each step) are treated precisely in Lemma B.2. in \cite{ALM}. For convenience, we list the correspondence between their notation and of this article in Table \ref{tab:LemB.1not}. 
\begin{table}
    \centering
    \begin{tabular}{ccccccccc}
        \cite{ALM}&  $\gamma$ & $c$ & $a$ & $q_1$ & $q_2$ & $b_1$ & $b_2$ &  $b_3$\\
        Our notation &  $\beta_2$ & $mh_1\hat h_2$ & $q_j^*\barr{\eta_2}$ & $u_{j+1}$ &$u_{j+2}$ & $q_1^*\barr{\eta_2}$ & $A_2h_1\barr{d_h\eta_2}$ & $A_1\hat h_2$\\
&&&&
    \end{tabular}
    \caption{Notation of Lemma B.2, \cite{ALM}}
    \label{tab:LemB.1not}
\end{table}

Then by Lemma B.2 of \cite{ALM}, we can bound $\CC_{j,m}$ by 
$$\CC_{j,m}\ll N^{\varepsilon},$$
unless 
\begin{align}
    &u_{j+1}\equiv u_{j+2}\bmod mh_1\hat h_2,~~\label{firstcong}\\
   \text{or}\ \ \ \ & q_1^*\pm A_2h_1\barr{d_h}(\barr{u_{j+1}}-\barr{u_{j+2}})\equiv 0 \bmod mh_1\hat h_2,\label{secondcong}\\
    \text{or}\ \ \ \ & q_1^*\equiv A_2h_1\barr{d_h}\barr{u_{j+1}}\bmod mh_1\hat h_2,\label{thirdcong}\\
   \text{or}\ \ \ \  &q_1^*\equiv A_2h_1\barr{d_h}\barr{u_{j+2}}\bmod mh_1\hat h_2,\label{fourthcong}
\end{align}

and in those cases, we can bound $$\cC_{j,m}\ll N^{\varepsilon}(mh_1\hat h_2)^{1/2}.$$
\begin{lemma}\label{S_jlemmabound}
Let $S_j$ be as defined in (\ref{Sjafterpoisson}). Then we have
    $$S_j\ll_{\varepsilon} \frac{N^{\varepsilon}Q^4Y_0^{3/2}H'^{1/2}}{M_0}\left(1+\frac{1}{QY_0^{1/2}}+Y_0^{1/2}H'^2+Q^{1/2}Y_0^{1/2}\right).$$
\end{lemma}
\begin{proof}
    Following the argument leading to (\ref{q_2*bound}), we get 
    $\frac{1}{B}\sum_{q_k^*\ll BN^{\varepsilon}}\ll N^{\epsilon}$ for $2\leq k \leq j$ and $\frac{v_2d_h}{Bv_1}\sum_{0\neq q_1^*\ll \frac{BN^{\varepsilon}v_1}{v_2d_h}}\ll N^{\varepsilon}$. 
    
    When none of the congruence relations (\ref{firstcong}), (\ref{secondcong}), (\ref{thirdcong}), (\ref{fourthcong}) holds, we have $\cC_j\ll N^{\varepsilon}$ along with the congruence relation in (\ref{Cjh2}). 
    In that case, following the definition of $\sum\cdots\sum$ (\ref{sumshorthand}) in $S_j$ (\ref{S_jstart}), and applying AM-GM inequality ($|\alpha_{j+1}\alpha_{j+2}|\ll |\alpha_{j+1}|^2+|\alpha_{j+2}|^2 $), we have  
    \begin{align}
        S_j\ll & \sum_{d_q\leq Q}\sum_{d_h\leq H'}\underset{(v_1,v_2)=1}{\sum_{v_1|d_q^{\infty}}\sum_{v_2|d_q^{\infty}}} \int\limits_{y_2\sim Y_0}\sum_{h_2\sim H'}\sum_{u_{j+2}\sim \frac{Q}{m_1d_qv_2d_h}}|\alpha_{j+2}|^2\sumstar_{x_0\bmod d_0}~\sumstar_{x_1\bmod d_1}~\sumstar_{x_2\bmod d_2}\nonumber\\
        &
    ~\sum_{h_1\sim \frac{H'}{d_h}}~\sum_{\tilde m\ll \frac{Q^2Y_0}{d_q}}\frac{v_2d_h}{Bv_1}\sum_{0\neq q_1^*\ll \frac{BN^{\varepsilon}v_1}{d_hv_2}}  \cdots \frac{1}{B}\sum_{q_{j-1}^*\ll BN^{\varepsilon}}{\sum_{\substack{u_{j+1}\sim \frac{Q}{m_1d_qv_2d_h}\\u_{j+1}\equiv u_{j+2}\bmod \eta_2}}}\cdot \frac{\sqrt{mh_1}}{\hat h_2^{2^{j-1}+1/2}}\int\limits_{\substack{y_1\sim Y_0\\|y_1-y_2|\ll M_0^{-1}}} \nn\\
    \ll& \sum_{d_q\leq Q}\sum_{d_h\leq H'}\underset{(v_1,v_2)=1}{\sum_{v_1|d_q^{\infty}}\sum_{v_2|d_q^{\infty}}} \int\limits_{y_2\sim Y_0}\sum_{h_2\sim H'}\sum_{u_{j+2}\sim \frac{Q}{m_1d_qv_2d_h}}|\alpha_{j+2}|^2 d_q\cdot \frac{H'^{3/2}}{d_h^{3/2}}\cdot \frac{Q^3Y_0^{3/2}}{d_q}\cdot \left(\frac{Q}{m_1d_qv_2d_hH'}+1\right)\frac{1}{M_0}\nn\\
    \ll& \frac{N^{\varepsilon}Q^3Y_0^{3/2}H'^{1/2}}{M_0}\left(Q+H'\right)
    \ll\frac{N^{\varepsilon}Q^4Y_0^{3/2}H'^{1/2}}{M_0}.\label{S_jcase0}
    \end{align}
   At the last line, we have used $H'\ll \frac{Q}{H}\ll Q.$
   
Now, we focus on the four exceptional cases, mentioned above. We start with (\ref{firstcong}) $u_{j+1}\equiv u_{j+2}\bmod mh_1\hat h_2$. We also have $u_{j+1}\equiv u_{j+2}\bmod \eta_2$ from $\cC_{j.h_2}$. Thus, we basically have 
$$u_{j+1}\equiv u_{j+2}\bmod mh_1h_2.$$
Then by AM-GM inequality, we have 
\begin{align*}
    &\sum_{u_{j+1}\sim \frac{Q}{m_1d_qv_2d_h}}\alpha_{j+1}\sum_{u_{j+2}\sim \frac{Q}{m_1d_qv_2d_h}} \alpha_{j+2}|\cC_j|
    \ll  \sum_{u_{j+1}\sim \frac{Q}{m_1d_qv_2d_h}}|\alpha_{j+1}|^2\left(\frac{Q}{m_1d_qv_2d_h\sqrt{mh_1h_2}}+\sqrt{mh_1\hat h_2}\right).
\end{align*}
Following the exact calculation of the first case (\ref{S_jcase0}), we get that 
\begin{align*}
    S_j\ll &\sum_{d_q\leq Q}\sum_{d_h\leq H'}\underset{(v_1,v_2)=1}{\sum_{v_1|d_q^{\infty}}\sum_{v_2|d_q^{\infty}}} \int\limits_{y_2\sim Y_0}\sum_{h_2\sim H'}\sum_{u_{j+2}\sim \frac{Q}{m_1d_qv_2d_h}}|\alpha_{j+2}|^2 d_q\cdot \frac{H'^{3/2}}{d_h^{3/2}}\cdot \frac{Q^3Y_0^{3/2}}{d_q}\\
    &\cdot \left(\frac{Q}{m_1d_qv_2d_h(QY_0^{1/2}H')}+(QY_0^{1/2}H')\right)\frac{1}{M_0}\\
    \ll & \frac{N^{\varepsilon}Q^4Y_0^{3/2}H'^{1/2}}{M_0}\left(\frac{1}{QY_0^{1/2}}+Y_0^{1/2}H'^2\right).
\end{align*}
In the case of (\ref{secondcong}), (\ref{thirdcong}) and (\ref{fourthcong}), we first derive that $h_1|q_1^*$. We also recall that $q_1^*\neq 0$. Now, we apply Cauchy's inequality on all the sums in $S_j$ to get $S_j\ll \sqrt{ S_{j,j+1}}\cdot \sqrt{ S_{j,j+2}}$, where 
\begin{align*}
    S_{j,j+1}:=&\sum\cdots \sum \cdot \sum_{u_{j+1}\sim \frac{Q}{m_1d_qv_2d_h}}|\alpha_{j+1}|^2 ~\frac{v_2d_h}{Bv_1}\sum_{0\neq q_1^*\ll \frac{BN^{\varepsilon}v_1}{d_hv_2}}  \cdots \frac{1}{B}\sum_{q_{j-1}^*\ll BN^{\varepsilon}}\\&\times \sum_{u_{j+2}\sim \frac{Q}{m_1d_qv_2d_h}}\cdot \frac{\sqrt{mh_1}}{\hat h_2^{2^{j-1}+1/2}}\cdot \frac{1}{B}\sum_{q_j^*\ll BN^{\varepsilon}}\cdot \CC_j,
\end{align*}
and $S_{j,j+2}$ is exactly the same with the role of $u_{j+1}$ and $u_{j+2}$ reversed. We will only do the last case (\ref{fourthcong}) as (\ref{secondcong}) and (\ref{thirdcong}) is similar. When (\ref{fourthcong}) holds, we can evaluate $S_{j,j+2}$ exactly similar to (\ref{S_jcase0}) and get 
$$S_{j,j+2}\ll \frac{N^{\varepsilon}Q^4Y_0^{3/2}H'^{1/2}}{M_0}(QY_0^{1/2}H'^{1/2}). $$
The extra term $(QY_0^{1/2}H'^{1/2})$ is coming from $\cC_{j,m}\ll \sqrt{mh_1}$. In $S_{j,j+1}$, we have 
\begin{align*}\frac{v_2d_h}{Bv_1}\sum_{0\neq q_1^*\ll\frac{Bv_1}{d_hv_2}}\sum_{u_{j+2}\sim \frac{Q}{m_1d_qv_2d_h}}|\CC_j|
\ll& \frac{v_2d_h}{Bv_1}\sum_{\substack{0\neq q_1^*\ll\frac{Bv_1}{d_hv_2}\\h_1|q_1^*}}\sum_{\substack{u_{j+2}\sim \frac{Q}{m_1d_qv_2d_h}\\m\hat h_2|u_{j+2}\\u_{j+2}\equiv u_{j+1}\bmod \eta_2}}\sqrt{mh_1}\\
\ll& \left(\frac{Q}{m_1d_qv_2d_hmh_2}+1\right)\cdot \frac{\sqrt{m}}{\sqrt{h_1}}.
\end{align*}
Then following the calculation in (\ref{S_jcase0}), we get 
\begin{align}
    S_{j,j+1}\ll &\sum_{d_q\leq Q}\sum_{d_h\leq H'}\underset{(v_1,v_2)=1}{\sum_{v_1|d_q^{\infty}}\sum_{v_2|d_q^{\infty}}} \int\limits_{y_2\sim Y_0}\sum_{h_2\sim H'}\sum_{u_{j+1}\sim \frac{Q}{m_1d_qv_2d_h}}|\alpha_{j+2}|^2 d_q\cdot \frac{H'^{3/2}}{d_h^{3/2}}\cdot \frac{Q^3Y_0^{3/2}}{d_qM_0}\nn\\
    &\cdot \left(\frac{Q}{m_1d_qv_2d_h(Q^2Y_0H')}+1\right)\cdot \frac{QY_0^{1/2}}{H'^{1/2}}\nn\\
    \ll&  \frac{N^{\varepsilon}Q^4Y_0^{3/2}H'^{1/2}}{M_0}\cdot \left(\frac{Q}{(Q^2Y_0H')}+1\right)\cdot \frac{Y_0^{1/2}}{H'^{1/2}}.\nn
\end{align}
Then we get 
\begin{align*}
    S_j\ll& S_{j,j+2}^{1/2}\cdot S_{j,j+1}^{1/2}
    \ll \frac{N^{\varepsilon}Q^4Y_0^{3/2}H'^{1/2}}{M_0}\cdot \left(1+Q^{1/2}Y_0^{1/2}\right). \qedhere
\end{align*}
\end{proof}
\subsection{Final calculations for main theorem}
Let us recall that 
$Y_0=C^2/QX^2N$. Hence, 
\begin{align*}
    &\left(1+\frac{1}{QY_0^{1/2}}+Y_0^{1/2}H'^2+Q^{1/2}Y_0^{1/2}\right)
    \ll \frac{N^{\varepsilon}N^{1/4}}{XQ^{1/2}K^{1/4}},
\end{align*}
provided $K\ll N^{1/3}$ and $H\gg \frac{N^{3/8}}{K^{5/8}}.$
So, from Lemma \ref{S_jlemmabound}, we have 
\begin{align}
    S_j\ll_{\varepsilon} \frac{Q^4Y_0^{3/2}H'^{1/2}}{M_0}\cdot \frac{N^{\varepsilon}N^{1/4}}{XQ^{1/2}K^{1/4}}\label{S-jfinall}
\end{align}
provided $K\ll N^{1/3}$ and $H\gg \frac{N^{3/8}}{K^{5/8}}.$
We already had $S_{1,1}\ll_{\varepsilon} \frac{Q^4Y_0^2H'^4}{M_0}$ (see \ref{S1,1bound}) and $S_{i,1}\ll \frac{Q^4Y_0^2H'^2}{M_0}$ (see \ref{S_j,1finalbound}) for all $ i\geq 2$. 
Hence, we have 
$$S_1\ll_{\varepsilon} H' \cdot \left(\frac{Q^4Y_0^2H'^2}{M_0}\right)^{\frac{2^{j-1}-1}{2^{j-1}}}\cdot \left(\frac{Q^4Y_0^{3/2}H'^{1/2}}{M_0}\cdot \frac{N^{\varepsilon}N^{1/4}}{XQ^{1/2}K^{1/4}}\right)^{2^{-j}},$$
for $q_1^*\neq 0$ and $u_2\neq u_3$. From Lemma \ref{u_2=u_3} and Lemma \ref{lemmas_1q=0} we get the bound on $S_1$ for the remaining cases and we put it into $\Delta$ (\ref{DeltaOneCauchyMain}). Then in the off diagonal case ($m\neq 0$), we get  
\begin{align}
    \Delta\ll_{\varepsilon} & \frac{M_0}{m_1^2}\cdot \frac{QY_0^{1/2}H'^{1/2}}{M_0^{1/2}}\cdot \left(\frac{Q^{3/2}Y_0^{1/2}H'^{1/2}}{M_0^{1/2}}+\frac{Y_0^{1/4}Q}{m_1^{1/2}M_0^{1/2}}\left(\frac{Q^{1/2}}{m_1^{1/2}H'^{1/2}}+H'^{1/2}\right)+\frac{Y_0Q^{2}H'^{3/2}}{M_0^{1/2}}\right)\nn\\
    \ll_{\varepsilon}& \left(\frac{Q^{5/2}Y_0H'}{m_1^2}+\frac{Q^{5/2}Y_0^{3/4}}{m_1^{5/2}}+\frac{Q^2Y_0^{3/4}H'}{m_1^2}+\frac{Y_0^{3/2}Q^{3}H'^{2}}{m_1^2}\right).
\end{align}
Finally, we calculate $S'$ (\ref{S'defn}) in the off-diagonal case. Once, we use the Ramanujan bound on average in the $m_1$ and $d$-sum,  we get 
\begin{align*}
    S'_{m\neq 0}\ll_{\varepsilon} &\frac{HN^{1+\varepsilon}}{C}\sup_{Q\ll C} \frac{1}{Q^2}\sup_{H'\ll \frac{QN^{\varepsilon}}{H}}\sum_{\substack{\frac{QCN^{\varepsilon}}{N}\leq X\leq C^{\varepsilon}\\\text{dyadic }} }\cdot \frac{X}{Y_0}\left(Q^{5/2}H'Y_0+Q^{5/2}Y_0^{3/4}+Q^2Y_0^{3/4}H'+Y_0^{3/2}Q^{3}H'^{2}\right)\\
    \ll_{\varepsilon}& \frac{N^{5/4+\varepsilon}}{K^{1/4}}+HN^{7/8+\varepsilon}K^{3/8}+N^{9/8+\varepsilon}K^{1/8}+\frac{N^{7/4+\varepsilon}}{K^{5/4}H}.
\end{align*}
For small values of $X$ (\ref{SXsmall}) and in the diagonal case (\ref{Sdiagonal}), we can bound $S'$ (\ref{S'defn}) by
$$\frac{N^{3/2+\varepsilon}}{K^{5/2}}+HN^{1/2+\varepsilon}K^{3/2}.$$
Hence, in total, we can bound $S'$ (\ref{S'defn}) by
\begin{align}
     \frac{N^{3/2+\varepsilon}}{K^{5/2}}+HN^{1/2+\varepsilon}K^{3/2}+\frac{N^{5/4+\varepsilon}}{K^{1/4}}+HN^{7/8+\varepsilon}K^{3/8}+N^{9/8+\varepsilon}K^{1/8}+\frac{N^{7/4+\varepsilon}}{K^{5/4}H}.\label{S'Finalcontribution}
\end{align}
As $K\leq N^{1/3}$ (otherwise $HN^{1/2}K^{3/2}>NH$), we observe that 
$$HN^{7/8}K^{3/8}\geq HN^{1/2}K^{3/2}\text{ and }\frac{N^{5/4}}{K^{1/4}}\geq N^{9/8}K^{3/8}.$$
Then optimally choosing $K=\frac{\sqrt{N}}{H}$, we observe that 
\begin{align}
    S'\ll_{\varepsilon} N^{\varepsilon}\left(N^{1/4}H^{5/2}+H^{5/8}N^{17/16}+N^{9/8}H^{1/4}\right),
\end{align}
which concludes the proof of Theorem \ref{maintheorem}.

\appendix
\section{Linearization}\label{appen.lin}
Let $V(x)$ be a smooth function supported in $[2,4]$ and let $\CX=(A+B)$ for some fixed real term $A$ and $B$ and let $a,b\in \bR$ and $a< 0$. 
\begin{lemma}\label{MellinBarnesLemma}
Let $s=c_1+it_1$ and $z=c_2+it_2$ for some $c_1,c_2>1$ 
    \begin{align}
        \CX^{b}\cdot V(\CX^aZ)=\frac{-1}{4\pi^2}\int\limits_{|t_1|\ll N^{\varepsilon}}\int\limits_{|t_2|\ll N^{\varepsilon}} \hat V(s)Z^{-s} \frac{\Gamma(as-b+z)\Gamma(-z)}{\Gamma(as-b)} A^{z}B^{-as+b-z}+O_A(N^{-A}).
    \end{align}
\end{lemma}
\begin{proof}
We will first apply Mellin inversion to get 
\begin{align}
  \CX^{b}\cdot V(\CX^aZ)=\CX^b\frac{1}{2\pi i}\int_{(c_1)}\hat V(s) \left(\CX^aZ\right)^{-s}ds=\frac{1}{2\pi i}\int_{(c_1)}\hat V(s)Z^{-s}\CX^{b-as}ds.\label{Mellin}
\end{align}
for some suitable $c_1>0$.

Now, we apply the Mellin-Barnes formula \cite[(1.44)]{DGSMellin} on $\CX=(A+B)$
\begin{align}
    \CX^{-\lambda}=(A+B)^{-\lambda}=\frac{1}{\Gamma(\lambda)}\frac{1}{2\pi i}\int_{(c_2)}\Gamma(\lambda+z)\Gamma(-z)A^{z}B^{-\lambda-z}dz.\label{MellinBarnes}
\end{align}
Applying (\ref{MellinBarnes}) into the expression (\ref{Mellin}) with $\lambda=as-b$, we get 
\begin{align}
    \frac{1}{2\pi i}\frac{1}{2\pi i}\int_{(c_2)}\int_{(c_1)}\hat V(s)Z^{-s} \frac{\Gamma(as-b+z)\Gamma(-z)}{\Gamma(as-b)} A^{z}B^{-as+b-z}.
\end{align}

Now that $A$ and $B$ are separated, we must regulate the imaginary term in the exponent of $A$ and $B$. At this moment, we denote $s=c_1+it_1$ and $z=c_2+it_2$. Then we have 
\begin{align}
&\frac{1}{4\pi^2}\int_{-\infty}^{\infty}\int_{-\infty}^{\infty}\hat V(c_1+it_1)Z^{-c_1-it_1} \frac{\Gamma((ac_1+c_2-b)+i(at_1+t_2))\Gamma(-c_2-it_2)}{\Gamma((ac_1-b)+iat_1)}\nn\\& \times A^{c_2+it_2}B^{-(ac_1+c_2-b)-i(at_1+t_2)}dt_1dt_2.
\end{align}
So, our aim is to prove that the double integral is negligibly small unless $t_1\ll N^{\varepsilon}$ and $t_2\ll N^{\varepsilon}$. 
If we look at the Mellin transform of $V$, we observe that 
$$\hat V(c_1+it_1)=\int_{0}^{\infty}x^{c_1-1+it_1}V(x) dx=\int_{0}^{\infty}x^{c_1-1}e(\frac{t_1}{2\pi}\log x)V(x)dx$$
By the first derivative bound, the integral above is negligibly small unless 
$$\frac{t_1\log x}{2\pi x}\ll N^{\varepsilon}\impliedby t_1\ll N^{\varepsilon}.$$ So, We have proved the first condition. 

By the Stirling's approximation, we have 
$$|\Gamma(c+it)|\sim \sqrt{2\pi }|t|^{c-1/2}e^{-\pi|t|/2}.$$
Then the $t_2$-integral $I_2$ is bounded by 
\begin{align}
    I_2\ll&\int_{-\infty}^{\infty}|\Gamma((ac_1+c_2-b)+i(at_1+t_2))||\Gamma(-c_2-it_2)|dt_2\nn\\
    \sim &2\pi \int_{-\infty}^{\infty} |at_1+t_2|^{ac_1+c_2-b-1/2}|t_2|^{-c_2-1/2}e^{-\frac{\pi}{2}(|at_1+t_2|+|t_2|)} dt_2.
\end{align}
We already had $t_1\ll N^{\varepsilon}$. Then if $t_2\gg N^{\varepsilon}$, we have $|t_2|\leq |at_1+t_2|\leq 2|t_2|$. 
Hence, 
\begin{align}
    I_2\ll &2\pi \int_{-\infty}^{\infty} |t_2|^{ac_1-b-1}e^{-\frac{\pi}{2}(|t_2|)} dt_2\ll N^{-A},
\end{align}
for any $A>0$. Thus, the double integral is negligibly small unless $t_1\ll N^{\varepsilon}$ and $t_2\ll N^{\varepsilon}$. 
\end{proof}
\section{Character sum analysis}\label{charsumappen}
We will follow the treatment of \cite{pal2023secondmomentdegreelfunctions}. 
We will first show that $d_q|m$ and $(\frac{m}{d_q}, v_1v_2u_1u_2)=1$. We open up the Kloosterman sums and get 
\begin{align}
    \CC=&\frac{1}{q_1q_2}\sumstar_{x_1\bmod q_1}e\left(\frac{\overline x_1\overline h_1}{q_1}\right)\sumstar_{x_2\bmod q_2}e\left(\frac{\overline x_2\overline h_2}{q_2}\right)\sum_{\beta \bmod q_1q_2}e\left(\frac{x_1\beta }{q_1}+\frac{x_2\beta }{q_2}+\frac{\beta m}{q_1q_2}\right)\nonumber\\
    =& \sumstar_{x_1\bmod q_1}e\left(\frac{\overline x_1\overline h_1}{q_1}\right)\sumstar_{x_2\bmod q_2}e\left(\frac{\overline x_2\overline h_2}{q_2}\right)\delta(x_1q_2+x_2q_1+m\equiv 0\pmod{q_1q_2}). \label{Charsum2}
\end{align}
Form the congruence relation, we note that $d_q|m$ and we have 
$$x_1u_2v_2+x_1u_1v_1\equiv -\frac{m}{d}\pmod{d_qv_1v_2u_1u_2}.$$
If we denote $\tilde m:=\frac{m}{d}$, we note that $(\tilde m, v_1v_2u_1u_2)$ because $(x_1,q_1)=1$ and $(x_2,q_2)=1$. Now, by Chinese Remainder Theorem, we can split the $\beta $-sum modulo $q_1q_2$ into three sums: $\beta_1\bmod u_1$, $\beta_2\bmod u_2$ and $\beta _3\bmod d_q^2v_1v_2$. We will also split the Kloosterman sums modulo $q_i\ (i=1,2)$ into two Kloosterman sums modulo $d_qv_i$ and $u_i$. Then the $\beta_1$-sum would be 
\begin{align*}
    &\sum_{\beta_1\bmod u_1}S(\overline h_1\overline{d_qv_1}, \beta_1\overline{d_qv_1}, u_1 )e\left(\frac{\beta_1m\overline{q_2v_1d_q}}{u_1}\right)
    = u_1e\left(-\frac{\overline m \overline h_1\overline{v_1}v_2u_2}{u_1}\right).
\end{align*}
Similarly, from the $\beta_2$-sum, we get 
\begin{align*}
    u_2e\left(-\frac{\overline m \overline h_2\overline{v_2}v_1u_1}{u_2}\right).
\end{align*}
Finally, from the $\beta _3$-sum, we get 
\begin{align*}
    &\sum_{\beta_3\bmod d_q^2v_1v_2} S(\overline h_1\overline {u_1}, \beta_3 \overline {u_1}, d_qv_1)S(\overline h_2\overline {u_2}, \beta_3 \overline {u_2}, d_qv_2)e\left(\frac{\beta_3m\overline{u_1}\overline{u_2}}{d_q^2v_1v_2}\right)\\
     =& d_q^2v_1v_2\underset{x_1v_2+x_2v_1+\tilde m\equiv0 \bmod d_qv_1v_2}{\sumstar_{x_1\bmod d_qv_1}\sumstar_{x_2\bmod d_qv_2}}e\left(\frac{\overline x_1\overline h_1u_2\overline {u_1}}{d_qv_1}\right)e\left(\frac{\overline x_2 \overline h_2u_1\overline{u_2}}{d_qv_2}\right).
\end{align*}
So we get 
\begin{align}
    \CC=e\left(-\frac{\overline m \overline h_1\overline{v_1}v_2u_2}{u_1}\right)e\left(-\frac{\overline m \overline h_2\overline{v_2}v_1u_1}{u_2}\right)\underset{x_1v_2+x_2v_1+\tilde m\equiv0 \bmod d_qv_1v_2}{\sumstar_{x_1\bmod d_qv_1}\sumstar_{x_2\bmod d_qv_2}}e\left(\frac{\overline x_1\overline h_1u_2\overline {u_1}}{d_qv_1}\right)e\left(\frac{\overline x_2 \overline h_2u_1\overline{u_2}}{d_qv_2}\right).
\end{align}
Now, by reciprocity, we can write 
$$e\left(-\frac{\overline m \overline h_1\overline{v_1}v_2u_2}{u_1}\right)=e\left(\frac{v_2u_2\overline{u_1}}{mh_1v_1}\right)e\left(-\frac{v_2u_2}{u_1v_1mh_1}\right).$$
As $\frac{v_2u_2}{u_1v_1mh_1}\ll 1$, we can second term above absorb into the smooth functions. So we write 

\begin{align}
    \CC=e\left(\frac{v_2u_2\overline{u_1}}{mh_1v_1}\right)e\left(\frac{v_1u_1\overline{u_2}}{mh_2v_2}\right)\underset{x_1v_2+x_2v_1+\tilde m\equiv0 \bmod d_qv_1v_2}{\sumstar_{x_1\bmod d_qv_1}\sumstar_{x_2\bmod d_qv_2}}e\left(\frac{\overline x_1\overline h_1u_2\overline {u_1}}{d_qv_1}\right)e\left(\frac{\overline x_2 \overline h_2u_1\overline{u_2}}{d_qv_2}\right)U(\cdots),
\end{align}
where $U(\cdots)$ is a smooth function. 
We recall that $d_q=(q_1,q_2)$. Then we can split $d_q$ into $d_q=d_0\cdot d_1\cdot d_2$ in the following manner
$$d_1=(d_q, v_1^{\infty}),\ d_2=(d_q,v_2^\infty)\text{ and }d_0=\frac{d_q}{d_1d_2}. $$
We note that $(d_0,v_1v_2)=1$. Then we can split the sum over $x_1$ and $x_2$ by CRT in the following fashion:
\begin{align}
    \sumstar_{x_1\bmod d_qv_1}(x_1)&\mapsto \sumstar_{x_{1,0}\bmod d_0}\sumstar_{x_{1,1}\bmod d_1v_1}\sumstar_{x_{1,2}\bmod d_2}\left(x_{1,0}\cdot \frac{d_qv_1}{d_0}\overline{\frac{d_qv_1}{d_0}}+x_{1,1}\cdot \frac{d_qv_1}{d_1v_1}\overline{\frac{d_qv_1}{d_1v_1}}+x_{1,2}\cdot \frac{d_qv_1}{d_2}\overline{\frac{d_qv_1}{d_2}} \right),\\
    \sumstar_{x_2\bmod d_qv_2}(x_2)&\mapsto \sumstar_{x_{2,0}\bmod d_0}\sumstar_{x_{2,1}\bmod d_1}\sumstar_{x_{2,2}\bmod d_2v_2}\left(x_{2,0}\cdot \frac{d_qv_2}{d_0}\overline{\frac{d_qv_2}{d_0}}+x_{2,1}\cdot \frac{d_qv_2}{d_1}\overline{\frac{d_qv_2}{d_1}}+x_{2,2}\cdot \frac{d_qv_2}{d_2v_2}\overline{\frac{d_qv_2}{d_2v_2}} \right).
\end{align}
With these notations, we can also split the congruence relation $$x_1v_2+x_2v_1+\tilde m\equiv0 \bmod d_qv_1v_2$$
into 
\begin{align}
    &x_{1,0}v_2+x_{2,0}v _1+\tilde m\equiv 0 \bmod{d_0}\implies x_{1,0}\equiv-\overline{v_2}(x_{2,0}v _1+\tilde m)\bmod{d_0},\\
    &x_{1,1}v_2+x_{2,1}v_1+\tilde m\equiv 0 \bmod {d_1v_1}\implies \overline{x_{1,1}}\equiv -v_2\overline{(x_{2,1}v_1+\tilde m)}\bmod {d_1v_1},\\
   \text{and }~~ &x_{1,2}v_2+x_{2,2}v_1+\tilde m\equiv 0 \bmod{d_2v_2}\implies \overline{x_{2,2}}\equiv -v_1\overline{(x_{1,2}v_2+\tilde m)}\bmod {d_2v_2}
\end{align}
We also note that $(\tilde m,v_1v_2)=1$. Then $(\tilde m, d_q)|d_0$. We recall that $\tilde m=m/d_q$. Now, we denote $d_m=(\tilde m,d_0^{\infty})$ and $\Fm=\frac{m}{d_qd_m}$. Now, we will split the character sum into the following four clusters of prime factors ($\Fm h_1h_2,~ d_0d_m,~ d_1v_1$ and $d_2v_2$):
\begin{align}
    \CC=\CC_h\cdot \CC_0\cdot \CC_1\cdot \CC_2
\end{align}
where $\CC_h$, $\CC_0$, $\CC_1$ and $\CC_2$ is stated and simplified below. We have
\begin{align}
    \CC_h=& e\left(\frac{u_2\overline{u_1}\cdot v_2 \overline{d_m d_qv_1}}{\Fm h_1}\right)e\left(\frac{u_1 \overline{u_2}\cdot v_1 \overline{d_md_qv_2}}{\Fm h_2}\right), 
\end{align}
and 
\begin{align}
    \CC_0=& e\left(\frac{u_2 \overline{u_1} v_2\overline{\frac{mh_1v_1}{d_0d_m}}}{d_0d_m}\right)e\left(\frac{u_1\overline{u_2}v_1\overline{\frac{mh_2v_2}{d_0d_m}}}{d_0d_m}\right)\underset{x_{1,0}\equiv-\overline{v_2}(x_{2,0}v _1+\tilde m)\bmod{d_0}}{\sumstar_{x_{1,0}\bmod d_0}\sumstar_{x_{2,0}\bmod d_0}} 
    e\left(\frac{\overline x_{1,0}\overline h_1u_2\overline {u_1}\overline{\frac{d_qv_1}{d_0}}}{d_0}\right)e\left(\frac{\overline x_{2,0} \overline h_2u_1\overline{u_2}\overline{\frac{d_qv_2}{d_0}}}{d_0}\right)\nonumber\\
    =&\sumstar_{x_{2,0}\bmod d_0}e\left(\frac{u_2 \overline{u_1} v_2\overline{\frac{mh_1v_1}{d_0d_m}}}{d_0d_m}\right)e\left(\frac{u_1\overline{u_2}v_1\overline{\frac{mh_2v_2}{d_0d_m}}}{d_0d_m}\right)
    e\left(-\frac{\overline{(x_{2,0}v _1+\tilde m)}v_2\overline h_1u_2\overline {u_1}\overline{\frac{d_qv_1}{d_0}}}{d_0}\right)e\left(\frac{\overline x_{2,0} \overline h_2u_1\overline{u_2}\overline{\frac{d_qv_2}{d_0}}}{d_0}\right)\nonumber\\
    =& \sumstar_{x_{0}\bmod d_0}e\left(\frac{u_2 \overline{u_1} v_2[1-d_m\overline{(x_{0}+d_m)}]\overline{\frac{mh_1v_1}{d_0d_m}}}{d_0d_m}\right)e\left(\frac{u_1\overline{u_2}v_1[1+\bar x_{0}d_m]\overline{\frac{mh_2v_2}{d_0d_m}}}{d_0d_m}\right)\nonumber\\
    =&  \sumstar_{x_{0}\bmod d_0}e\left(\frac{u_2 \overline{u_1} v_2\overline{(1+\bar x_0 d_m)}\overline{\frac{mh_1v_1}{d_0d_m}}}{d_0d_m}\right)e\left(\frac{u_1\overline{u_2}v_1[1+\bar x_{0}d_m]\overline{\frac{mh_2v_2}{d_0d_m}}}{d_0d_m}\right).
\end{align}
In the second last line, we have changed the variable from $x_{2,0}\mapsto \Fm \cdot x_{0}$. We get 
\begin{align}
    \CC_1=&e\left(\frac{v_2u_2\overline{u_1}\overline{\frac{mh_1}{d_1}}}{d_1v_1}\right)\underset{\overline{x_{1,1}}\equiv -v_2\overline{(x_{2,1}v_1+\tilde m)}\bmod {d_1v_1}}{\sumstar_{x_{1,1}\bmod d_1v_1}\sumstar_{x_{2,1}\bmod d_1}}e\left(\frac{\overline x_{1,1}\overline h_1u_2\overline {u_1}\cdot \overline{\frac{d_q}{d_1}}}{d_1v_1}\right)e\left(\frac{\overline x_{2,1} \overline h_2u_1\overline{u_2}\cdot \overline{\frac{v_2d_q}{d_1}}}{d_1}\right)\nonumber\\
    =&\sumstar_{x_{2,1}\bmod d_1}e\left(\frac{v_2u_2\overline{u_1}\overline{\frac{mh_1}{d_1}}}{d_1v_1}\right)e\left(\frac{-v_2\overline{(x_{2,1}v_1+\tilde m)}\overline h_1u_2\overline {u_1}\cdot \overline{\frac{d_q}{d_1}}}{d_1v_1}\right)e\left(\frac{\overline x_{2,1} \overline h_2u_1\overline{u_2}\cdot \overline{\frac{v_2d_q}{d_1}}}{d_1}\right)\nonumber\\
    =& \sumstar_{x_{1}\bmod d_1}e\left(\frac{v_2u_2\overline{u_1}(1-\overline{(x_{1}v_1+1)})\overline{\frac{mh_1}{d_1}}}{d_1v_1}\right)e\left(\frac{u_1\overline{u_2} \overline x_{1} \cdot \overline{\frac{mh_2v_2}{d_1}}}{d_1}\right)\nonumber\\
    =& \sumstar_{x_{1}\bmod d_1}e\left(\frac{v_2u_2\overline{u_1}\cdot \overline{(\bar x_{1}+v_1})\overline{\frac{mh_1}{d_1}}}{d_1}\right)e\left(\frac{u_1\overline{u_2} \overline x_{1} \cdot \overline{\frac{mh_2v_2}{d_1}}}{d_1}\right).
\end{align}
In the second last line, we have changed $x_{2,1}\mapsto x_{1} \tilde m $. 
Similarly, we get 
\begin{align}
    \CC_2=&  \sumstar_{x_{2}\bmod d_2}e\left(\frac{u_1\overline{u_2}\cdot v_1\cdot\overline{\bar x_2+v_2} \overline{\frac{mh_2}{d_2}}}{d_2}\right)e\left(\frac{u_2\overline{u_1} \overline x_{2} \cdot \overline{\frac{mh_1v_1}{d_2}}}{d_2}\right)
\end{align}
Hence, we can write 
\begin{align}
    \CC=\sumstar_{x_0\bmod d_0}\sumstar_{x_1\bmod d_1}\sumstar_{x_2\bmod d_2}e\left(\frac{\overline{u_1}u_2A_1}{mh_1}\right)e\left(\frac{u_1\overline{u_2}A_2}{mh_2}\right)
\end{align}
where 
\begin{align}
    A_1\equiv \begin{cases}
        v_2\overline{v_1}&\pmod {\Fm h_1}\\
        v_2\overline{v_1}\overline{[1+\bar x_0d_m]}&\pmod {d_0d_m}\\
        v_2\cdot  \overline{\bar x_1+v_1}& \pmod {d_1}\\
        \overline{v_1x_2}&\pmod {d_2}
    \end{cases}~~~~~~\text{and}~~~~~~~
     A_2\equiv \begin{cases}
        v_1\overline{v_2}&\pmod {\Fm h_2}\\
        v_1\overline{v_2}[1+\bar x_0d_m]&\pmod {d_0d_m}\\
        \overline{v_2x_1}&\pmod {d_1}\\
        v_1 \cdot \overline{\bar x_2+v_2}& \pmod {d_2}
    \end{cases}.
\end{align}
\section*{Acknowledgement}
The authors are grateful to Ritabrata Munshi for many helpful remarks and discussions.
\printbibliography
\end{document}